\documentclass{article}
\usepackage[utf8]{inputenc}
\usepackage{fullpage}
\usepackage{grffile}
\usepackage[position=b]{subcaption}

\usepackage{amsmath,amsfonts,amssymb,amsthm}
\usepackage{graphicx}
\usepackage{comment}
\newtheorem{lemma}{Lemma}
\usepackage{braket}

\newcommand{\AL}[1]{{\bf AL:  #1}}

\newcommand{\bkappa}{\boldsymbol{\kappa}}
\newcommand{\R}{\mathbb R}
\newcommand{\Z}{\mathbb Z}

\newcommand{\C}{\mathbb C}
\newcommand{\br}{\mathbf r}

\newcommand{\bR}{\mathbf R}
\newcommand{\bT}{\mathbf T}
\newcommand{\bk}{\mathbf k}
\newcommand{\bh}{\mathbf h}

\newcommand{\bK}{\mathbf K}

\newcommand{\BZ}{\mathcal B}

\newcommand{\be}{\begin{equation}}
\newcommand{\ee}{\end{equation}}
\newcommand{\ba}{\begin{eqnarray}}
\newcommand{\ea}{\end{eqnarray}}

\author{Ivan Duchemin$^{1}$, Luigi Genovese$^{1}$, Eloïse
  Letournel$^{2}$, Antoine Levitt$^{2}$, Simon Ruget$^{1}$}

\title{Efficient extraction of resonant states in systems with defects}
\begin{document}

\maketitle

\begin{abstract}
  We introduce a new numerical method to compute resonances induced by
  localized defects in crystals. This method solves an integral
  equation in the defect region to compute analytic continuations of
  resolvents. Such an approach enables one to express the resonance in
  terms of a ``resonance source'', a function that is strictly localized within the defect region.
  The kernel of the integral equation, to be applied on such a source term, is the Green
  function of the perfect crystal, which we show can be computed efficiently by a complex deformation of the Brillouin
  zone, named Brillouin Complex Deformation (BCD), thereby extending to reciprocal space the concept of complex coordinate transformations.
\end{abstract}
\footnotetext[1]{Univ. Grenoble Alpes, CEA, IRIG-MEM-L Sim, 38054 Grenoble,
France}
\footnotetext[2]{Inria Paris and Universit\'e Paris-Est, CERMICS, Ecole des Ponts ParisTech, Marne-la-Vall\'ee, France}


\section{Introduction}
Hamiltonians of non-homogeneous quantum systems generally have a spectrum consisting
of two parts: discrete bound states, and continuous scattering states.
The bound states are localized in the region around the scatterer, and can therefore be captured
numerically by standard discretization methods in a computational box which is sufficiently large.
Scattering states, on the other hand, are
delocalized, and their efficient numerical approximation requires discretization techniques which are able to 
capture long-range oscillations. These differences
complicate the description of response phenomena that depend sensitively on
continuum states, such as scattering cross sections and resonances. In
this paper, we will consider methodologies to compute such properties
in one-body Hamiltonians of independent particles; nonetheless, the motivation and the expected
applications come from mean-field models such as time-dependent
density functional theory (TDDFT).

As a concrete example, consider a one-body (possibly mean-field)
Hamiltonian $\hat H$, describing for instance a molecule or a defect in a
solid. Many time-dependent response properties can be described by
sums of functions of the type
\begin{align}
  \label{eq:spectral_measures}
  f(E) = \lim_{\eta \to 0^{+}} \langle  \psi_{0}| (E+i\eta-\hat H)^{-1} | \psi_{0}\rangle
\end{align}
where $\psi_{0}$ can be expressed via a localized function. These functions probe the
continuous spectrum at energy $E$: mathematically, we have
${\rm Im}(f(E)) = -\pi \langle \psi_{0}| \hat{d\mu_{H}}(E) |
\psi_{0} \rangle$, where $\hat {d\mu_{H}}$ is the projection-valued spectral
measure associated to $\hat H$. When the Hamiltonian can be
interpreted as a small perturbation of a reference Hamiltonian that has both bound and
continuous states at the same energy $E$ (for instance, a Hamiltonian
where a molecule is surrounded by infinitely high potential barriers,
or where a crystalline defect is disconnected from the host crystal),
the coupling between these states typically results in a bump of
$f$ near $E$. This corresponds to a resonance, which can be formally
defined as a pole in the analytic continuation of $f$ from
the upper complex plane into the lower \cite{dyatlov2019mathematical}.

Resonant states can be found in many areas of Quantum Physics. 
Also referred to as ``Gamow vectors'' or ``Siegert states'',  
they can be defined for isolated systems as solutions of the time-independent Schr\"odinger equation 
subject to outgoing boundary conditions.
Described at first by Gamow \cite{Gamow} via quasi-stationary states, the concept of resonant states has been
widely developed in the field of atomic and nuclear physics
(see e.g.\ Ref.~\cite{PhysRevC.47.768}), then adopted for the analysis of scattering properties of quantum systems with open boundaries~\cite{Hatano2008}.
Various literature has shown that Siegert states \emph{encode}
in compact form the response properties of a system \cite{Myo01051998,Lind1994}.
In particular, the analytic structure of the resolvent operator (i.e. the Green's function) is completely 
determined by resonant energies and wavefunctions.
In the words of Ref.~\cite{PhysRevLett.79.2026}, resonant states expansions offer the ``possibility of a unified
description of bound states, resonances, and continuum spectrum in terms of a purely discrete set of states''.
For one-body Hamiltonians of quantum systems, the identification of resonant energies and wavepackets would unlock
efficient computational aproaches for perturbation theory which would preserve the physico-chemical meaning of the configuration interaction space.

From the perspective of the extraction of physical observables that can be compared to experiments,
the computation of functions of the form \eqref{eq:spectral_measures},
and a fortiori of their analytic continuation to the lower $E$ complex
plane, is challenging. This is because the truncation of $\hat H$ to a finite
region of space will discretize the energy spectrum, in which case the
definition of $f$ above is very singular. This is related to
qualitative differences in wave propagation described by the
full and truncated Hamiltonians. In unbounded domains, waves propagate
to infinity at large times; this results in the correlations $\langle
\psi_{0}| e^{-i\hat H t} | \psi_{0}\rangle$  decaying to
zero in time, and therefore $f$ is a smooth function of energy. In
bounded domains however, standing waves form at discrete energies and
correlation functions do not decay in time. 

Practical computations can
be performed by approximating $f(E)$ by $f(E+i\eta)$ for
some finite $\eta > 0$, which acts as an artificial dissipation
parameter, homogeneous to an inverse time. This makes possible the
computation of $f(E+i\eta)$ by choosing a computational domain
whose size must be large compared to the mean free path of the waves
(proportional to $1/\eta$ times the group velocity of the waves). The resulting scheme requires a delicate
balance between $\eta$ and the size of this computational domain.
Furthermore, all analytic information in the lower-half complex plane
is lost in this procedure, and the computation of resonances requires
a potentially numerically unstable extrapolation from the upper
complex plane to the lower.

More sophisticated numerical schemes have been developed, sometimes
with different names depending on their communities of origin. A first
variation of the dissipation method, known as complex absorbing potential \cite{MUGA2004357}, is
to replace $\hat H$ by $\hat H + \eta \hat V_{\rm CAP}(x)$, where $\hat V_{\rm CAP}(x)$ is
non-zero only outside of a central region. This has similar properties
to the above-mentioned technique of a uniform $\eta$, but does not
modify the operator in the defect region, and can therefore be
preferable in practice.

A second technique is to exploit the analytic continuation of the
solution to complex numbers to transform scattering or resonant states
into localized states \cite{complex_scaling_3D}. This is done by
replacing the space variable $x$ by $x e^{i\theta}$. In this complex
scaling approach, the continuous spectrum of the non self-adjoint
operator $\hat H_{\theta}$ is moved from $\R^{+}$ to $e^{-2i\theta} \R^{+}$,
allowing for the computation of functions $f$ as above as well as
direct computation of resonances. This can either be done on the whole
space, or only on the exterior of a central region; the resulting
scheme is then called exterior complex scaling, or perfectly matched
layers.

Finally, a third class of methods is to eliminate the degrees of
freedom outside of a computational domain. At the discrete level, this
is done by a Schur complement approach, while at the continuous level
this corresponds to Dirichlet-to-Neumann (DTN) maps
\cite{givoli2013numerical}. This is especially attractive in simple
geometries (such as one-dimensional or three-dimensional in radial
coordinates), where the DTN corresponds to a simple local mixed boundary
condition.

None of these schemes are fully satisfactory, in particular for the study of defects in quantum systems. The ``black-box''
approach of adding an artificial dissipation  requires a complicated convergence
study with respect to the dissipation parameter $\eta$ (the so-called
$\eta$ trajectories) and may necessitate very large computational domains to
obtain stable results. 
Both the complex scaling as well as the DTN approaches require a
particular form for the operator outside of the computational domain,
with the DTN method even requiring the analytic solution of the
equation. This is particularly problematic in the case where $\hat H$
is not fully homogeneous in space outside of a central region, but
only periodic, as is the case for defects in solids. Since the
exterior problem can be solved in something resembling a closed form
only on half-spaces, one has to resort
to a matching procedure that makes it impossible to compute resonances
directly \cite{bonnetbendhia:hal-01793511} in dimensions greater than one. In particular, the
computation of the DTN map for a periodic operator on the outside of
an arbitrary rectangular domain seems to be a computationally intractable
problem.

In this paper, we introduce an integral equation formalism which
bypasses such difficulties by
an \textit{algebraic} rather than \textit{geometric} splitting:
instead of considering the interior and the exterior problem
separately, we rather split the Schrödinger operator as
$\hat H = \hat H_{0} + \hat V$, where $\hat V$ is localized only in the central region.
This reformulates the problem as an integral equation posed in the
central region, similar to the Lippman-Schwinger method used in
scattering problems.

The kernel of this integral equation is given by the Green function
$R_0(\cdot,\cdot;z)$ of the Hamiltonian $\hat H_{0}$ of the periodic
crystal, possibly extended the lower complex plane of energies $z$.
The Green function can be expressed as an integral over the Brillouin
zone, which we deform using a multi-dimensional generalization of the
Cauchy integral formula. To our knowledge, this is the first numerical
method able to do so in the multidimensional case. We choose the
deformation function $\bk\mapsto\bk + i\bh(\bk)$ so that the
singularities of $R_{0}$ get pushed down into the lower complex
plane for $z$, extending the domain of validity of the integral
formula to the continuation of $R_{0}$. These complex variable techniques have long been used in theoretical studies
\cite{gerard,Hoang2011TheLA,joly:hal-00977852}. A similar
approach has recently been used as a numerical method in 1D scattering problems
\cite{zhang2021numerical}. We demonstrate in the Appendix A that it
can be interpreted as a natural generalization of the complex scaling method
to non-parabolic dispersion relations. The resulting scheme only requires unit
cell computations, and proves very efficient in practice.

The outline of the paper is as follows. First, we introduce the
reformulation to an integral equation in a general setting in
Section~\ref{sec:method}. Then we apply it to the model case of a
local perturbation of the free Laplacian, and compare it to
an established method (complex scaling) in Section~\ref{sec:free}. We
introduce the Brillouin zone Complex Deformation (BCD) algorithm in
Section~\ref{sec:periodic_media}, and apply it to example problems
(one-dimensional chain and two-dimensional graphene) in
Section~\ref{sec:apps}. In the Appendix A, we reinterpret the BCD method
as a generalization of the complex scaling method.



\section{Method}
\label{sec:method}
\subsection{Statement of the problem}
We consider the abstract setting of a Hamiltonian
\begin{align*}
 \hat H = \hat H_{0} + \hat V
\end{align*}
where the underlying Hilbert space can be discrete or continuous. We
assume that $\hat H_{0}$ has a (continuous or discrete) translation
invariance (which makes the wavenumber a good quantum number for the eigenstates of $\hat H_0$), and that $\hat V$ is localized in a region of space. In
particular, we will consider three typical cases:
\begin{enumerate}
\item Molecular hamiltonians, where $\hat H_{0} = \hat T$ represents the kinetic operator, and is self-adjoint
  on $L^{2}(\R^{d})$, and $\hat V$ is a (local or nonlocal) potential
  modeling the Coulomb attraction by nuclei, plus possibly Hartree or
  exchange-correlation terms at the mean-field level.
\item Crystal hamiltonians, where $\hat H_{0} = \hat T + \hat V_{\rm per}+V$ where
  $\hat V_{\rm per}$ is periodic in space, representing the background
  potential created by a perfect crystal, and $\hat V$ is the potential
  created by a crystallographic point defect.
\item Tight-binding models, which can be thought of as simplifications
  of the above crystalline model (e.g. through expansion of the orbitals in
  a basis of Wannier functions). Here the state space is discrete,
  with $M$ degrees of freedom at each point of a discrete lattice: the
  underlying Hilbert space is isomorphic to $\ell^{2}(\Z^{d}, \mathbb{C}^{M})$. $\hat H_{0}$ is periodic in the sense that for each $\bR, \bR' \in
  \Z^{d}$, the $M \times M$ matrix $H_{0}(\bR, \bR')$ satisfies the
  property $H_{0}(\bR, \bR') = H_{0}(\bR+\bT, \bR'+\bT)$ for all
  $\bT \in \Z^{d}$. Here also the (local or nonlocal) potential
  models a point defect.
\end{enumerate}
By suitable modifications of the definitions of $\hat H_{0}$ and $\hat V$, the
above setup can be easily modified to accomodate cases that do not fit
exactly (see for instance Section~\ref{sec:graphene} for the example
of an adatom in graphene). We assume that $\hat H$ is real (although the
extension to complex Hamiltonians poses no difficulty).

\medskip

Since $\hat H$ is self-adjoint, its resolvent
\begin{align*}
  \hat R(z) = (z-\hat H)^{-1}
\end{align*}
is well-defined as a bounded operator when ${\rm Im}(z) > 0$. Its
kernel $R(\br,\br';z)$ is localized, in the sense that
$R(\br,\br',z) \to 0$ when $|\br-\br'| \to \infty$. This well-known
fact is referred to as Combes-Thomas estimates in the mathematical
physics literature \cite{combes_asymptotic_1973}. In fact, this decay
is exponential, with a decay rate related to the imaginary part of $z$.

As $z$ approaches the
spectrum of $\hat H$, the resolvent diverges; in fact,
\begin{align*}
  \|\hat R(z)\| = \sup_{\|\ket \psi\| = 1} \|\hat R(z)\ket \psi\| = \frac 1 {{\rm dist}(z,\sigma(\hat H))}
\end{align*}
where $\sigma(\hat H)$ is the spectrum of $\hat H$ \cite{reed1979i}. This is linked to
delocalization of the kernel of $\hat R(z)$. However, the kernel of $\hat R(z)$
can sometimes be continued analytically beyond the real axis \cite{dyatlov2019mathematical,gerard}; it is
then typically exponentially growing. Accordingly, if $\ket \psi$ is an
arbitrary localized test vector, then the function
\begin{align*}
  \bra \psi \hat R(z) \ket \psi ,
\end{align*}
defined as a holomorphic function on the upper complex plane, can
sometimes be extended analytically through the essential spectrum into the lower complex plane. This
analytic continuation can have poles, which are typically insensitive
to the choice of $\ket \psi$. We call such a pole $z_0$ a resonance energy,
and the projector on the resonant states can be identified from the
residue of $\hat R$ around $z_0$. The resonant states are delocalized solutions of
the Schrödinger equation with a complex energy; the imaginary part of $z_{0}$
is inversely proportional to their lifetime.

\subsection{Computation of resonances}
\label{sec:computation_of_resonances}
To compute resonances, we first use the following identity among operators:
\begin{align}
  \label{eq:resolvent_identity}
  \hat R(z) = \hat R_{0}(z) (\hat 1-\hat V \hat R_{0}(z))^{-1}
\end{align}
for $z$ in $\mathbb{C}^+$, where $\hat R(z) = (z-\hat H)^{-1}$ and $\hat R_{0}(z) = (z-\hat H_{0})^{-1}$. This identity
is variously known as a resolvent identity, the Dyson equation, or the
Duhamel formula, and can be seen as an operator version of the
Lippmann-Schwinger equation. 
If $\hat R_{0}(z)$ can be extended analytically to the lower complex plane of $z$,
then resonances can be found by solving the equation
\begin{align}
  \label{eq:phi}
  \boxed{\ket{\varphi(z)} = \hat V \hat R_{0}(z) \ket{\varphi(z)}}
\end{align}
for $z$ and $\ket {\varphi(z)}$. The main interest of this formulation is that, by construction,
the ``resonance source'' term $\ket {\varphi(z)}$ is localized in the defect region where the support of $\hat V$
lies, which is not the case of the resonant vector $\ket{\psi(z)}=\hat R_{0}(z) \ket{\varphi(z)}$, a delocalized
eigenvector of $\hat H = \hat H_{0} + \hat V$ (the resonant state).

Consider now a simple resonance at $z_{0}$ (such that the
dimension of the kernel of $\hat 1 - \hat V \hat R_{0}(z_{0})$
is $1$). As $\ket {\varphi(z_{0})}$ is not directly associated to a physical state,
its normalization is not well defined. However, it
is useful to normalize it such that the resolvent has the asymptotic
form
\begin{align}
  \label{eq:res_expansion}
  \hat R(z) \approx \frac 1 {z-z_{0}} |\psi(z_0)\rangle\langle \overline{\psi(z_0)}|
\end{align}
close to $z_{0}$. 


From calculations detailed in Appendix B, this holds as long as 
\begin{align}
  \langle \overline{\psi(z_0)}| \hat V \hat R_{0}'(z_{0})| \varphi(z_0) \rangle = -1
\end{align}
This condition fixes the magnitude and phase of $\ket \varphi$ and
$\ket \psi$ up to sign.


To find resonances, we need to (a) discretize $\ket \varphi$; (b)
compute the action of the analytic continuation of $\hat R_{0}(z)$;
and (c) solve the equation \eqref{eq:phi}. The first task (a) is
standard since $\ket \varphi$ is localized, and can be done using
any method used to compute ground-state properties. Task (c) takes the
form of a nonlinear eigenvalue problem $A(z) x = 0$. For small
systems, one can simply solve for ${\rm det}(A(z)) = 0$. For larger
systems where determinants might be ill-conditioned or hard to
compute, one can use nonlinear eigensolvers (for instance, applying
Newton's method to the set of equations $A(z) x = 0, \|x\|=1$). 
Solving nonlinear eigenproblems is a well-studied topic with a variety of efficient algorithms \cite{guttel_tisseur_2017}. The
task we will focus on is therefore (b), the computation of the analytic
continuation of $\hat R_{0}(z)$.

\section{Free Laplacian}
\label{sec:free}
\subsection{Theory}
Before tackling periodic problems, we consider a
simple one-dimensional model where $\hat R_{0}$ is explicit. We
emphasize that this case can be
treated using other methods than the one presented here (see for
instance \cite{complex_scaling_3D}), and is
simply presented as a test bed for understanding the methodology.
The Hamiltonian is the following:
\begin{align*}
    \hat H=\hat T+\hat V,
\end{align*}
and we assume that $\hat V$ is a localized potential.

Let us first study the situation where the unperturbed Hamiltonian is
$ \hat H_0 = -\Delta $, with resolvent kernel $R_{0}$. For
${\rm Im}(z) > 0$, consider the equation
\begin{align*}
  (z+\Delta) S(\br) = \delta_{0}(\br)
\end{align*}
Together with the boundary condition that $S$ must go to zero at
infinity, this equation admits a single solution for $z$ in the upper complex plane:
\begin{align*}
  S_1(\br) = \frac{e^{i \sqrt{z} |\br|}}{2i \sqrt{z}},
\end{align*}
where we choose the convention
${\rm arg}(\sqrt z) \in (-\tfrac \pi2,\tfrac \pi 2)$. For $z$ in the
lower complex plane, the only localized solution would be 
\begin{align*}
     S_2(\br) = -\frac{e^{-i \sqrt{z} |\br|}}{2i \sqrt{z}}
\end{align*}
Note that the definition of the Green function differs between the
upper and lower complex planes. When $z$ is continued from the upper
to the lower complex plane through the negative real axis, $S_{1}$ and
$S_{2}$ match; however, when the continuation is performed through the
positive real axis (the spectrum of $\hat H_{0}$), they differ. This gives $S$ the structure of a
multivalued function, a Riemann surface of the same type as the square
root function, with a singularity at $0$.


From the above considerations, for ${\rm Im}(z) > 0$ we have
\begin{align*}
  R_{0}(\br, \br'; z ) = \frac{e^{i \sqrt{z} |\br-\br'|}}{2i \sqrt{z}}
\end{align*}
with an artificial branchcut on $\mathbb{R}^{-}$ and a singularity
(branchpoint) at 0.  This formula can be continued across the positive
real axis (where the kernel explodes as a function of $|\br-\br'|$). When $\hat V$ is compactly supported, using the identity
\eqref{eq:resolvent_identity} together with the analytic Fredholm
theory, one can show (Theorem 2.2 of \cite{dyatlov2019mathematical})
that $R$ extends to a meromorphic function (holomorphic except for a
countable set of points at which $R$ has finite-order poles, the
resonances) across the real axis. Using the method in
Section~\ref{sec:computation_of_resonances}, the equation
\eqref{eq:phi} for $\ket \varphi$ becomes a Fredholm integral equation
\begin{align}
  \label{eq:int_eq}
  \varphi(\br) = \frac{1}{2i \sqrt{z}} V(\br) \int_{\R} e^{i \sqrt{z} |\br-\br'|} \varphi(\br') d\br' 
\end{align}
for $z$ in the lower complex plane. As a shorthand we have omitted the explicit dependence of $\varphi$ on $z$.


\subsection{Application}
We
illustrate on the example of a double-well potential $ V(x)=
2(e^{-(\frac{x}{2})^2}- e^{-x^2})$ as shown in Figure~\ref{fig:potential}, for which we expect resonances
localized in the middle of the well.
\begin{figure}[h!]
    \centering
    \includegraphics[width=0.4\textwidth]{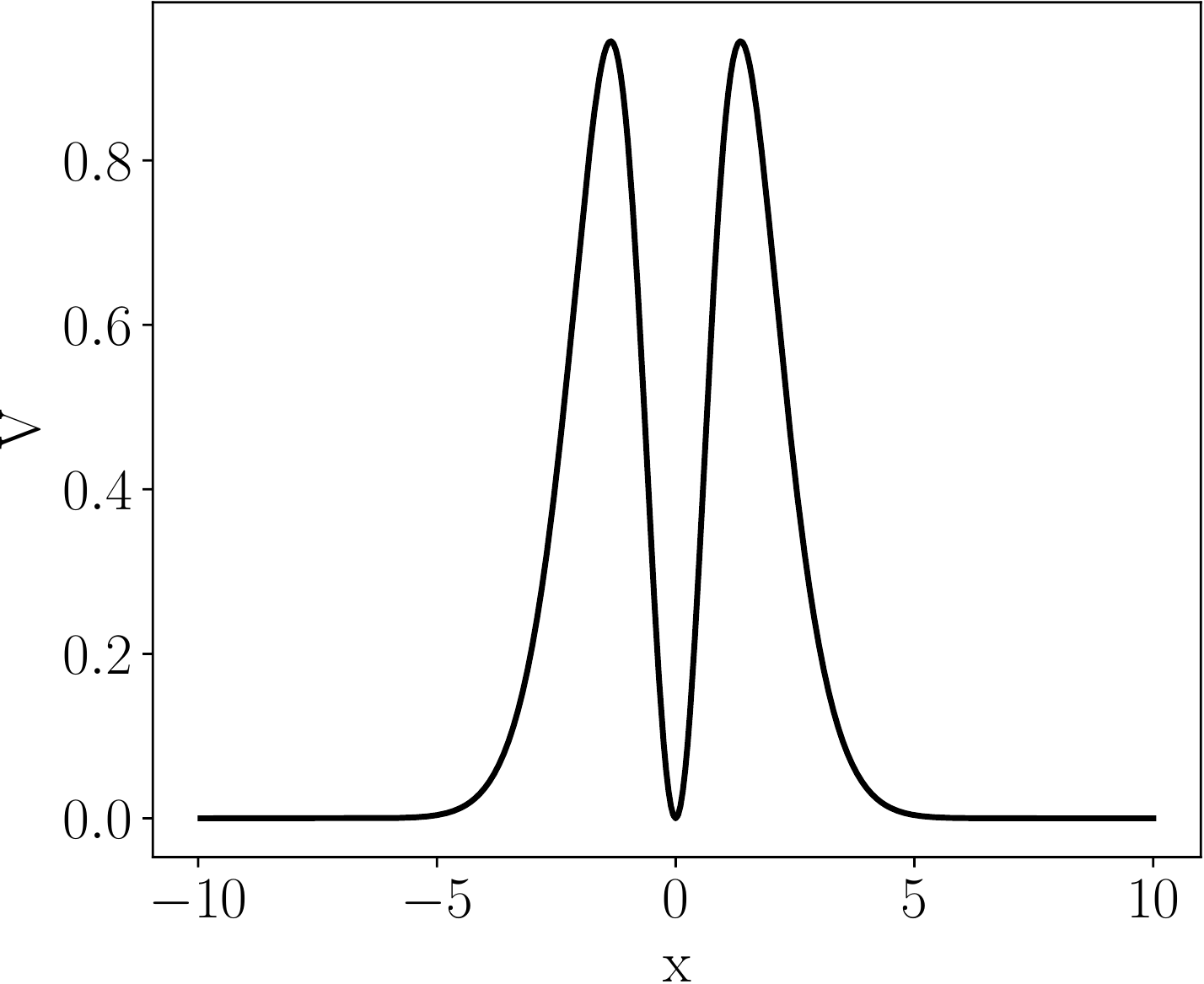}
    \caption{Potential $V(x)$.}
    \label{fig:potential}
\end{figure}

To acknowledge the capability of our method to capture physically relevant states, we compare our results to the uniform complex scaling method, which rotates
the Hamiltonian in the complex plane and looks for eigenvalues of the
non-hermitian operator,
\begin{align*}
    \hat H_{\theta}=& \hat S_{\theta}\hat H \hat S_{\theta}^{-1}  = -e^{2i\theta}\Delta + V(x e^{i\theta})
\end{align*}
The complex scaling operator $\hat S_{\theta}$ is given by the real space representation
\begin{align*}
  (S_{\theta}f) (x)& =f(xe^{i\theta})
\end{align*}
This transformation leaves the discrete spectrum and the resonances of
$\hat H$ invariant, but rotates the continuous spectrum of $\hat H$ by
an angle $-2\theta$. Eigenvalues in the lower complex plane which used
to be on the lower Riemann surface of the Green function can thus
appear as eigenvalues of the rotated Hamiltonian. For our comparison
we discretize $\hat H_{\theta}$ using simple finite differences with
step $h$ in a domain $[-\tfrac L2, \tfrac L2]$, and use $\theta=\pi/5$
throughout. For our method, we solve the integral equation
\eqref{eq:int_eq} using finite differences on the same grid.

\begin{figure}[h!]
\centering
\begin{subfigure}[t]{0.45\textwidth}
\centering
    \includegraphics[width=\textwidth]{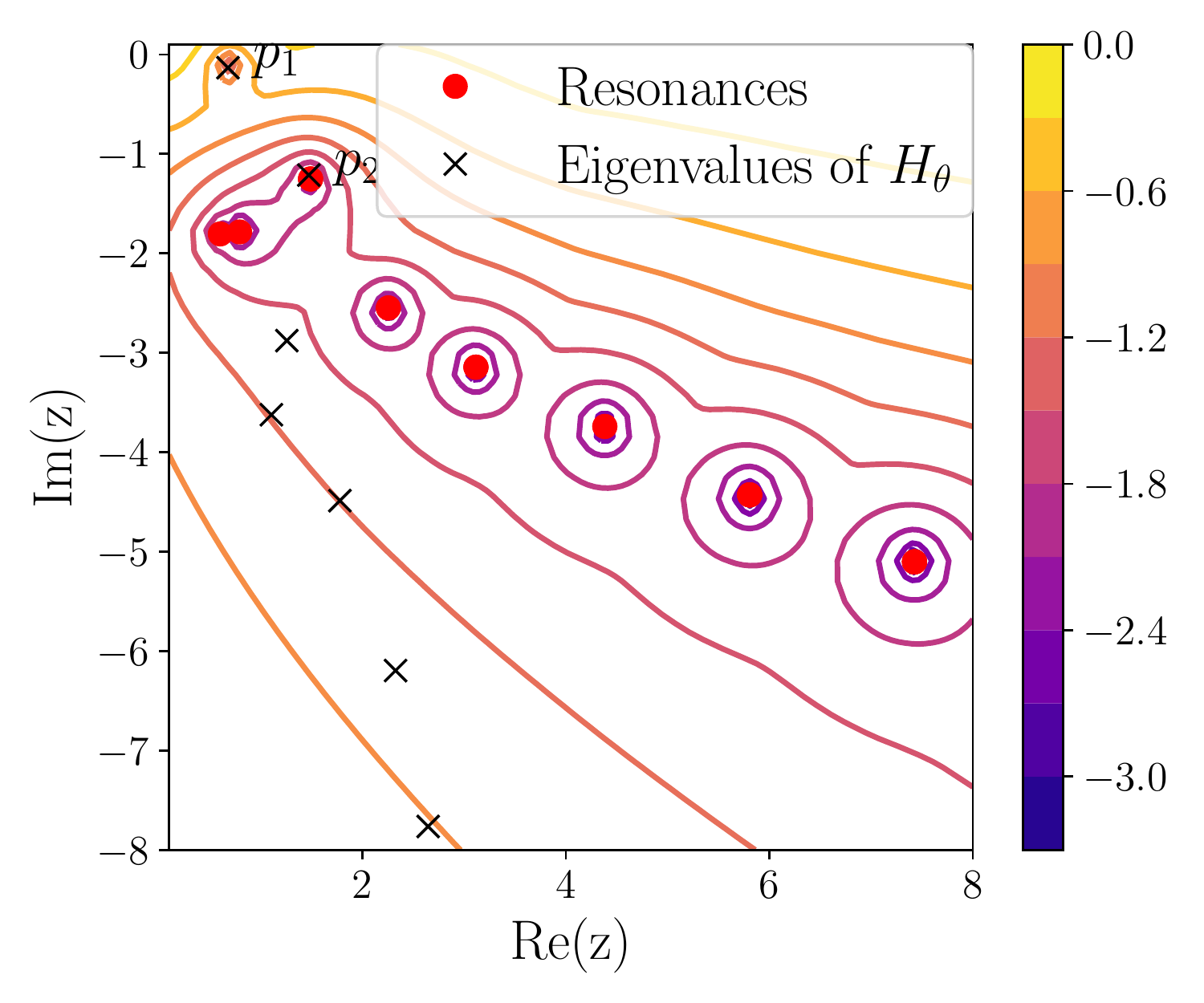}
    \caption{Resonances for $L=10$, $h=0.05$}
    \end{subfigure}
    \begin{subfigure}[t]{0.45\textwidth}
    \centering
    \includegraphics[width=\textwidth]{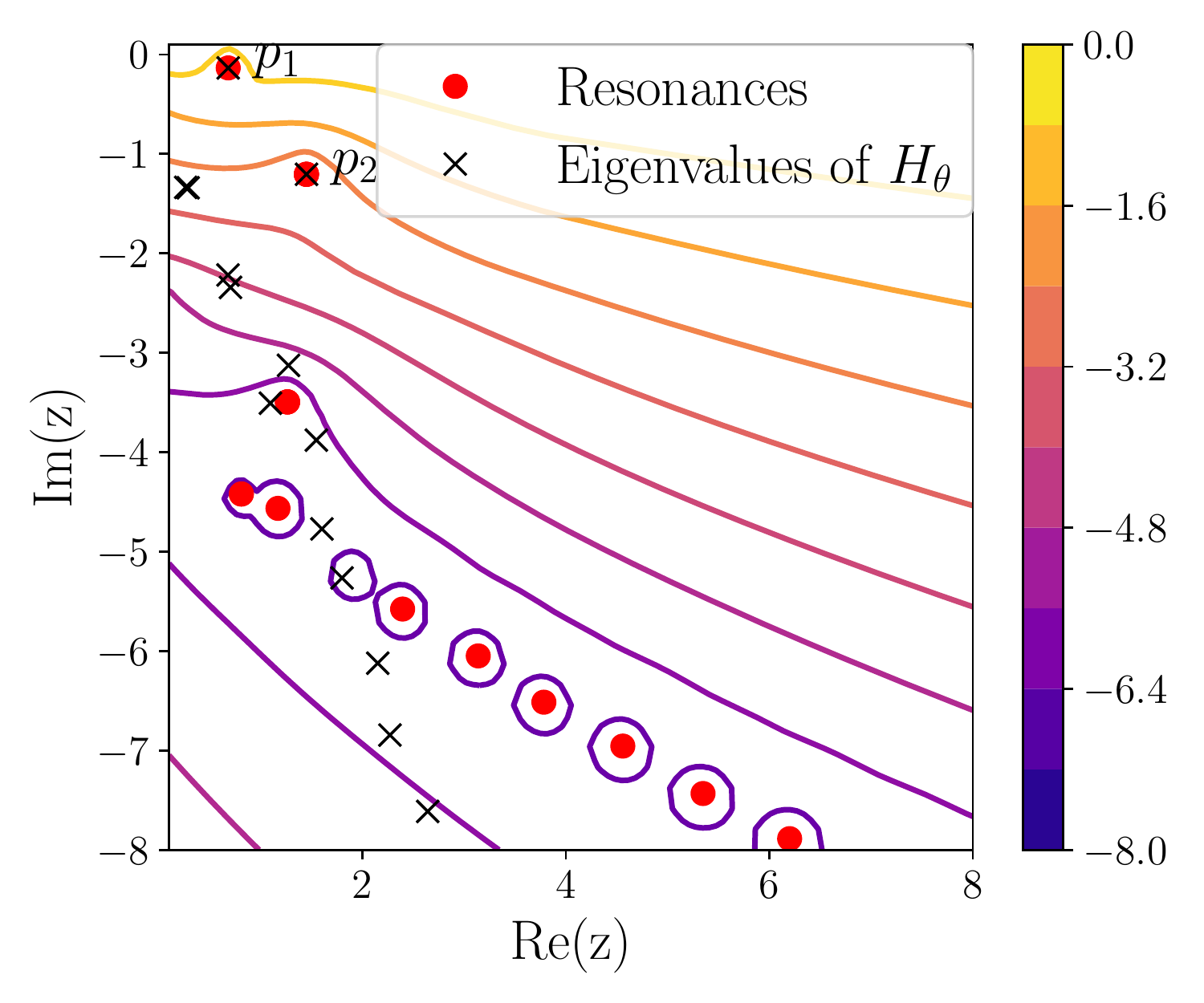}
    \caption{Resonances for $L=20$, $h=0.05$. }
    \end{subfigure}
\caption{ Base 10 logarithm of the smallest  singular value of $1 -
  \hat V \hat R_{0}(z)$, and eigenvalues of the complex scaled $\hat
  H_{\theta}$ with $\theta=\pi/5$. Resonances as well as spurious poles are visible for both methods. The spurious poles do not converge with $L$, whereas the resonances do.}
\label{fig:green_and_CS}
\end{figure}

We display our results Figure~\ref{fig:green_and_CS}. For our method,
we plot the smallest singular value of the discretization of
$\hat 1 - \hat V \hat R_{0}(z)$, which is zero at resonances.
In Figure~\ref{fig:green_and_CS}, two types of poles and
eigenvalues can be seen. For the two closest to the real axis, both
methods give the same location, and this location is stable as the box
size $L$ is increased. Lower poles are spurious. Those given by our method narrow and plummet more steeply when $L$ goes to infinity. With complex scaling, the angle of the line is stable with $L$, but the poles tighten up as well. The well-known complex virial theorem~\cite{complex_scaling_3D} can be employed to distinguish spurious poles from actual resonances without recurring to multiple calculations with varying $L$.

\begin{figure}[h!]
  \captionsetup[subfigure]{justification=centering}
\centering
 \begin{subfigure}[t]{\textwidth}
        \centering
        \includegraphics[width=0.49\textwidth]{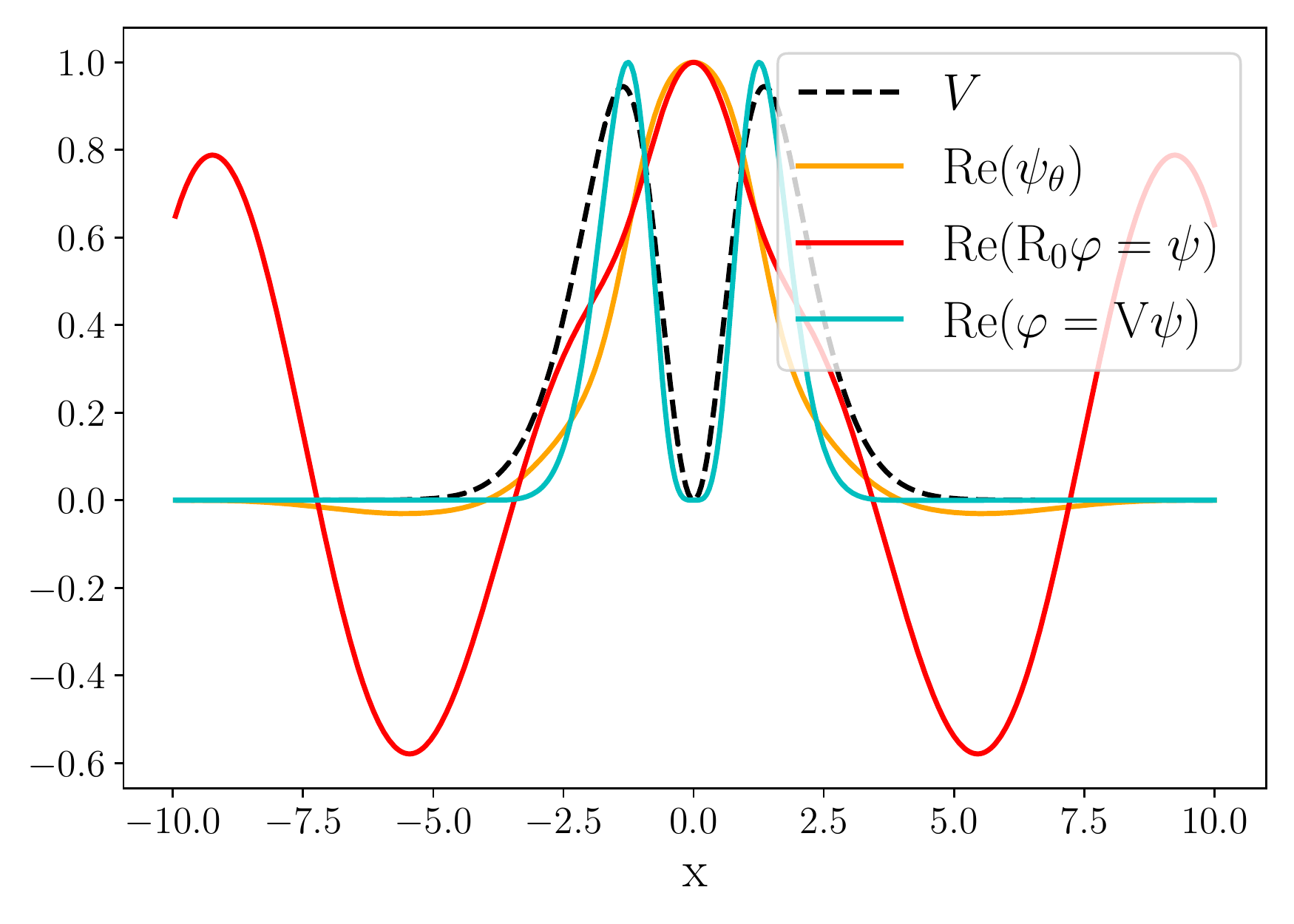}
        \includegraphics[width=0.49\textwidth]{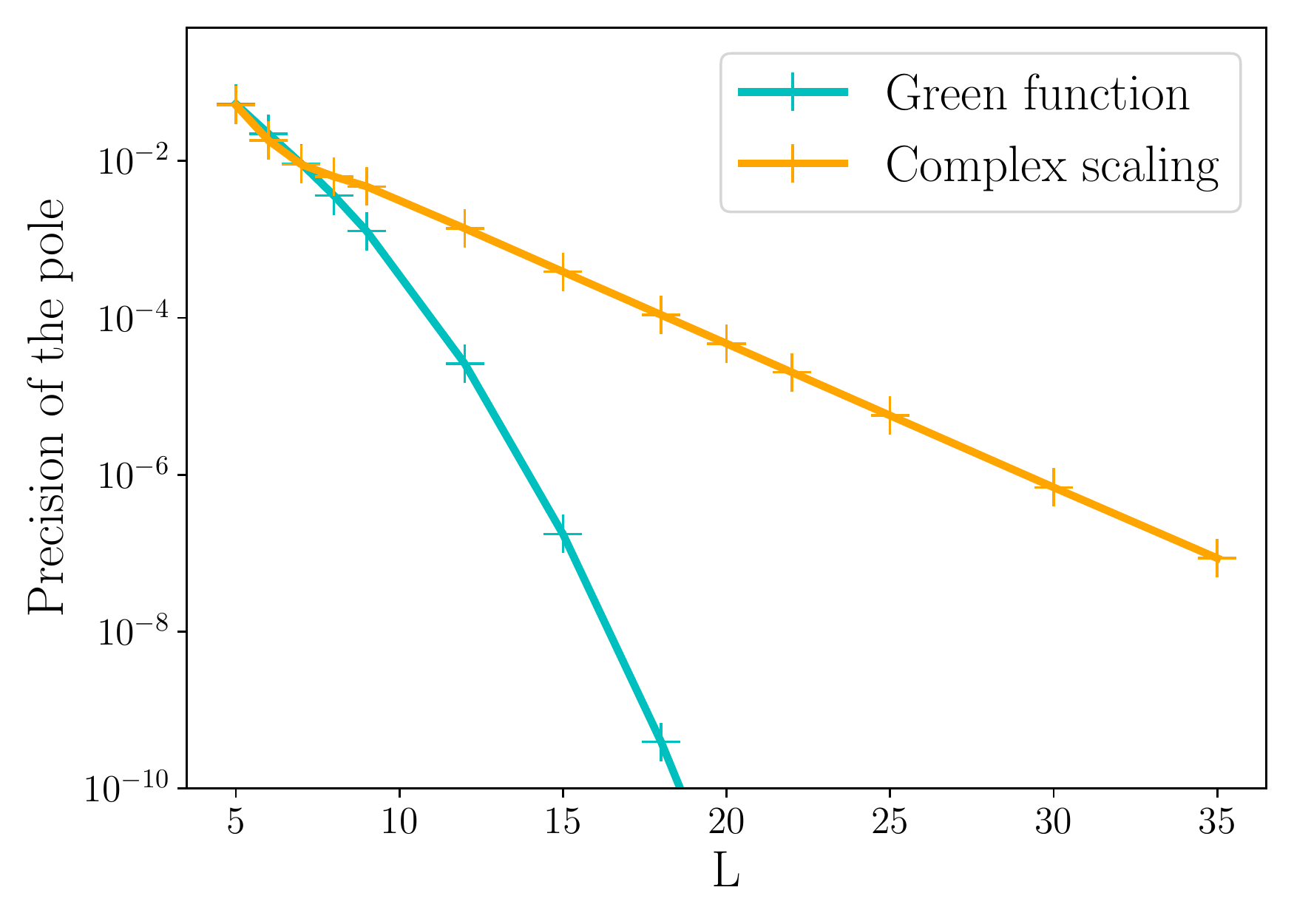}
    \caption{Results for $p_1$.}
    \medskip
  \end{subfigure} 
    \begin{subfigure}[t]{\textwidth}
        \centering
        \includegraphics[width=0.49\textwidth]{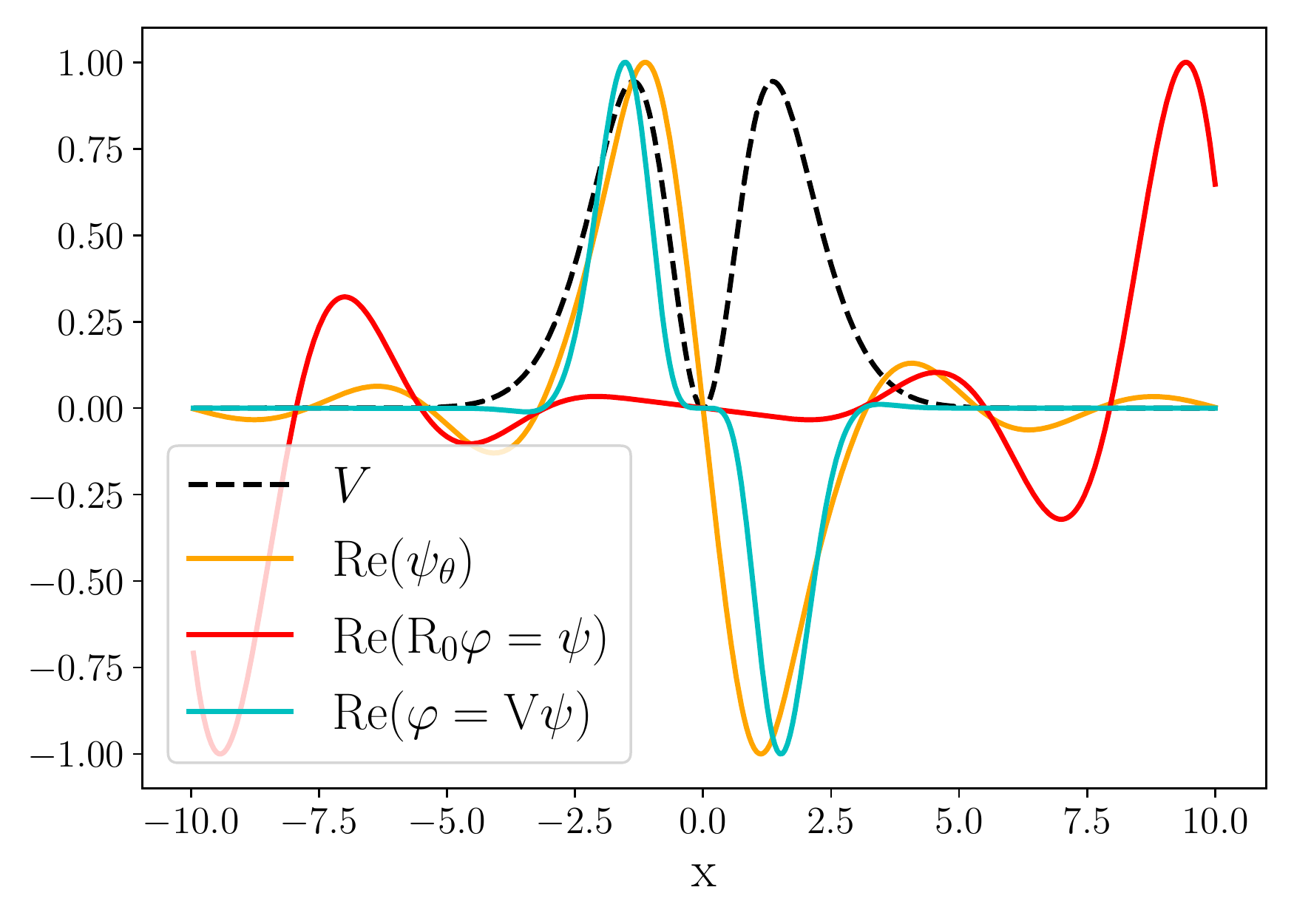}     
        \includegraphics[width=0.49\textwidth]{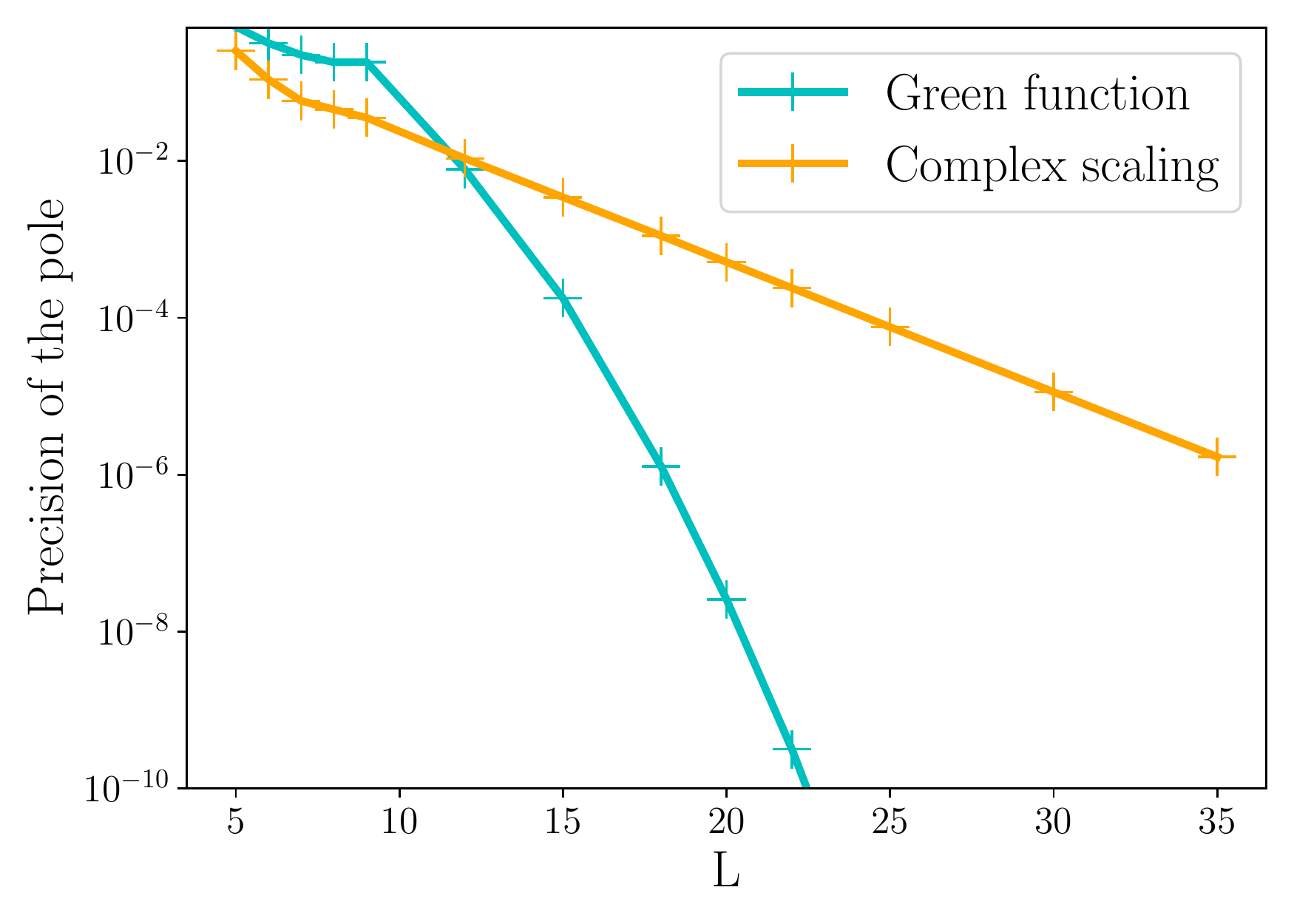}
      \caption{Results for $p_2$.}
    \end{subfigure}

     \caption{Resonant states and convergence rates for the two first
       poles $p_{1}$ and $p_{2}$, for the two methods ($h = 0.05$).
       Left panel: potential (black dashed line), resonant state
       $\psi$ (red), $\varphi = V \psi$ (blue) and rotated resonant
       state $\psi_{\theta}$ (yellow). Right panel: convergence of the
       pole location as a function of the box size. The normalization
       of the states is arbitrary and done for plotting clarity.
       \\
     }
    \label{fig:compare_eigenstates}
\end{figure}

We focus on the two most shallow resonances , denoted by $p_{1} \approx
0.68- 0.13 i$ and $p_2\approx 1.45-1.21i$ in Figure~\ref{fig:green_and_CS}.
We display in Figure~\ref{fig:compare_eigenstates} the function
$\varphi$ associated to the resonant source term of the eigenvector with eigenvalue 0 for the operator $\hat 1-\hat V\hat R_0$ at
the pole, as well as the resonant function $\psi$.
$\varphi$ is localized on the support of $V$, and $\psi$ diverges
(slowly since ${\rm Im}(z) = 0.13$) at infinity.  We also show the
eigenvector $\psi_{\theta}$ of the rotated Hamiltonian (obtained with
complex scaling) associated to the resonance. 

To estimate the computational feasibility of our
approach, we have also compared in Figure~\ref{fig:compare_eigenstates} the convergence rate of the resonant energy $p_{1}$ as a
function of the box size. The grid size $h$ is kept fixed. At finite
$h$, both methods converge to a slightly different value of the
resonance energy, and convergence is assessed relative to the fixed-$h$
value. The asymptotic convergence profile is related to the decay of the respective
objects discretized: Gaussian-like for the Green function method,
which discretizes $\phi = V \psi$, and exponential for the complex scaling
method, which discretizes $\psi_{\theta}$. This gives a better asymptotic convergence for our method than for the complex scaling. 
The performances of our method in small simulation boxes are equivalent or slightly below those obtained by complex scaling.
We however caution here that our aim is not a direct comparison of the performances of the two methods, which is a case-study-dependent
investigation that goes beyond the scope of our paper. It is enough for our purposes to show that the computational complexity of our method is comparable to that of other equivalent approaches.

To summarize, such an approach enables us to \emph{directly} access to the resonant wavefunction and energies,
without the need to transform the Hamiltonian in complex space, and by working on a computational domain which coincides with the support of the potential.
In order to do that, we need to express $R_0$ (otherwise stated the Green function of the Helmholtz Equation) for $z$ in the lower half of the complex plane
in the desired region, and filter the spurious poles with usual techniques, like the complex virial theorem.



\section{Periodic media}
\label{sec:periodic_media}
We now consider the case where
\begin{align*}
  \hat H = \hat H_{0} + \hat V
\end{align*}
with $\hat H_{0}$ a periodic operator and $\hat V$ a localized
potential. Although our method applies to continuous models of the form
$-\Delta + V_{\rm per}(\br) + V(\br)$, with $V_{\rm per}$ being a function having the periodicity
of a lattice, we will present the method using a discrete
tight-binding model (discrete Schrödinger operator) where $\Gamma$ is
a $d$-dimensional lattice. This is done to minimize numerical issues
related to the discretization of the unit cell, and to simplify the
exposition. We refer to the conclusion for perspectives in applying
our method to continuous Hamiltonians.

The tight-binding models we consider have $M$ internal degrees of
freedom per lattice site $\bR \in \Gamma$, with a lattice $\Gamma$ isomorphic to
$\Z^{d}$. The Hilbert space is
$\ell^{2}(\Gamma, \mathbb{C}^{M})$ with wavefunctions indexed by the lattice
site $\bR \in \Gamma$ and the degree of freedom $i \in \{1, \dots,
M\}$. For each $\bR, \bR' \in \Gamma$, $H_{0}(\bR,\bR')$ is a $M \times
M$ matrix satisfying
\begin{align*}
  H_{0}(\bR, \bR') = H_{0}(\bR+\bT, \bR'+\bT)
\end{align*}
for all $\bT \in \Gamma$.
\subsection{Green function of a periodic Hamiltonian}
\label{sec:green_function_periodic_hamiltonian}

Since $\hat  H_{0}$ is periodic, we can label its states as Bloch waves
$\psi_{n\bk}(\bR) = e^{i\bk \cdot \bR} u_{n\bk}(\bR)$ where
$u_{n\bk} \in \mathbb{C}^{M}$ is normalized \cite{alloul}. The index $\bk$ enumerates the Brillouin
zone $\BZ$, a unit cell of the reciprocal lattice $\Gamma^{*}$ (the set
of $\bK \in \R^{d}$
such that $\bK \cdot \bT$ is a multiple of $2\pi$ for all $\bT \in
\Gamma$). The $u_{n\bk}$ are the orthonormal solutions of

\begin{align} \label{eq:bloch0}
   H_{0,\bk} \ket {u_{n\bk}} = \varepsilon_{n\bk} \ket{u_{n\bk}}
\end{align}
where the reciprocal-space kernel $H_{0,\bk} \in \mathbb{C}^{M \times M}$ is the Fourier transform of $H_{0}({\bf 0},\cdot)$:
\begin{align} \label{eq:definitionH}
  H_{0,\bk} = \sum_{\bT \in \Gamma} e^{i\bk \cdot \bT} H_{0}({\bf 0}, \bT).
\end{align}
Conversely, we have
\begin{align*}
  H_{0}(\bR, \bR') = \frac 1 {|\BZ|} \int_{\BZ} e^{i\bk \cdot (\bR- \bR')} H_{0, \bk} d\bk
\end{align*}
and therefore, when ${\rm Im}(z) > 0$,
\begin{align}
  \label{eq:R0_per}
  R_{0}(\bR, \bR';z) &= \frac 1 {|\BZ|} \int_{\BZ} e^{i\bk \cdot (\bR- \bR')}\frac 1 {z - H_{0,\bk}} d\bk\\
   &= \frac 1 {|\BZ|} \int_{\BZ} e^{i\bk \cdot (\bR- \bR')}\sum_{n=1}^{M} \frac{u_{n\bk} u_{n\bk}^*}{z - \varepsilon_{n\bk}} d\bk.
\end{align}
When ${\rm Im}(z) > 0$, this function can be
computed simply by discretizing the integral. This is usually done by
a uniform sampling of the Brillouin zone (Monkhorst-Pack grid \cite{monkhorst}), a
simple quadrature that is very efficient because it integrates
explicitly low-order Fourier harmonics and is therefore exponentially
accurate for analytic functions \cite{trefethen}. However, as the imaginary
part of $z$ decreases to zero, the integrand is more and more
singular, with singularities concentrating on the Fermi surface $S({\rm Re}(z))$, where
\begin{align*}
  S(E) &= \cup_{n=1,\dots,M}S_{n}(E) = \big\{\bk \in \BZ, \exists n \in \{1,\dots,M\}, \varepsilon_{n\bk} = E\big\} \\
  S_{n}(E) &= \big\{\bk \in \BZ, \varepsilon_{n\bk} = E\big\}
\end{align*}
The integration therefore gets less accurate as ${\rm Im}(z)$ is
reduced. The analytic structure of $R_{0}$ is also destroyed by
the discretization: the discretized function becomes meromorphic on
$\mathbb{C}$, with a finite number of poles at the eigenvalues
$\varepsilon_{n\bk}$. In particular, the continuation of $R_{0}$ for ${\rm
Im}(z) < 0$ cannot be obtained. Nevertheless, the continuation exists
around all energies $E \in \R$ such that the following two conditions
hold:
\begin{itemize}
\item Bands do not intersect at energy $E$: $\forall \bk \in S(E),
  \varepsilon_{n\bk} = \varepsilon_{m\bk} \Rightarrow n=m$;
\item The group velocities are nonzero at energy $E$: $\forall \bk \in
  S(E), \varepsilon_{n\bk} = E \Rightarrow \nabla \varepsilon_{n\bk}
  \neq 0$.
\end{itemize}
Under these conditions, the Fermi surface $S(E)$ is a union of smooth
surfaces, and the resolvent can be analytically continued near $E$
\cite{cances:hal-01796582,gerard}. The discrete set of energies $E$ for which one of
the conditions above do not hold generally result in branch point
singularities, and are called \textit{van Hove singularities}. These
conditions make intuitive sense. The resolvent being analytic at a
particular energy means that wave propagation at that energy is
``regular'': Bloch waves have a well-defined non-zero group velocity.
The presence of a van Hove singularity can be interpreted as a
resonance (anomalous wave propagation) on the real axis. Analytic continuation 
of the resolvent near these singularity points is not possible \cite{gerard} or physically meaningful
(because the anomalous wave propagation caused by the van Hove singularity effectively hides any possible resonance).
Our method is therefore only relevant in the energy regions distant from van Hove singularities.

Taking into account these constraints, we realize the analytic continuation of
$R_0$ in the lower complex plane of $z$ by applying a
complex coordinate transformation in the \emph{reciprocal} space
\begin{align}
\bk \to {\mathbf \bkappa}(\bk)= \bk + i\bh(\bk)\;,
\end{align}
which does not modify the position of the \emph{singular} points of the
transformed dispersion relation $\varepsilon_{n \mathbf \kappa}$.
In other terms, $\bkappa (\bk)$ is defined such as to approach the identity transformation in the vicinity of a van Hove singularity.
We call this transformation a Brillouin Complex Deformation (BCD). We now study
how to perform the continuation in practice.

\subsection{Complex deformation of the Brillouin zone}
\label{sec:periodic_integral_computation}
To remedy the two problems of the naive quadrature (loss of accuracy
near the real axis, and loss of analytic continuation), we will deform
the integration domain into the complex plane. This is easier to see
in one dimension, where the integral \eqref{eq:R0_per} can be
deformed using the Cauchy formula. Remarkably, this can be extended to higher dimensions:
\begin{lemma}
\label{lemma:contour_deformation}
  Let $I(\bk)$ be a $\Gamma^{*}$-periodic function, analytic in an
  open set $U = \R^{d} + i [-\eta,\eta]^{d}$. Then, for all $\Gamma^{*}$-periodic and continuously differentiable functions
  $\bh(\bk) : \R^{d} \to [-\eta,\eta]^{d}$, we have
  \begin{align*}
    \int_{\BZ} I(\bk) d\bk = \int_{\BZ} I(\bkappa(\bk)) \det(\bkappa'(\bk)) d\bk = \int_{\BZ} I(\bk + i\bh(\bk)) \det(1 + i\bh'(\bk)) d\bk
  \end{align*}
\end{lemma}
\begin{proof}
  Consider for $\alpha \in \mathbb{C}$ the function
  \begin{align*}
    J(\alpha) = \int_{\BZ} I(\bk+\alpha \bh(\bk)) \det(1 + \alpha \bh'(\bk)) d\bk,
  \end{align*}
  analytic for $|\alpha| \le 1$. For $\alpha$ real, positive and
  sufficiently small, we have by a change of variables $J(\alpha) = J(0)$. By analytic
  continuation, it follows that
  $J(i) = J(0)$.
\end{proof}
Applied to the integrand
\begin{align*}
  I_{z}(\bk) = \frac 1 {|\BZ|} e^{i\bk \cdot (\bR- \bR')}\frac1 {z-H_{0,\bk}},
\end{align*}
this formula provides an explicit representation of the analytic
continuation of $R_{0}(\bR,\bR';z)$ from the upper complex plane to
the lower. More explicitly, taking a path $z(t)$ originating from a point in the upper complex plane and descending into the lower
plane, the formula
\begin{align*}
  R_{0}(\bR,\bR';z) = \frac 1 {|\BZ|} \int_{\BZ} I_{z}(\bkappa(\bk)) \det(\bkappa'(\bk)) d\bk
\end{align*}
provides a continuation of $R_{0}(\bR,\bR';z)$ along the path as 
long as the eigenvalues of $H_{0,\bk+i\bh(\bk)}$ for $\bk \in \BZ$
do not intersect the path. Appropriate choices of $\bh$ can
therefore extend the region of validity of the integral formula into
the lower complex plane.
Then, the deformed integral can be discretized in a standard
Monkhorst-Pack grid:
\begin{align*}
  R_{0}(\bR,\bR';z) &\approx \frac {|\BZ|} {N^{d}} \sum_{\bk \in \BZ_{N}} I_{z}(\bkappa(\bk)) \det(\bkappa'(\bk)) d\bk\\
  &= \frac 1 {N^{d}} \sum_{\bk \in \BZ_{N}}  e^{i(\bk+i\bh(\bk)) \cdot (\bR- \bR')}\frac1 {z-H_{0,(\bk+i\bh(\bk))}} \det(1 + i \bh'(\bk)) d\bk
\end{align*}
where $\BZ_{N}$ is the set of $N^{d}$ points in the Monkhorst-Pack grid.



The task at hand is now to choose $\bh(\bk)$ so that
$\varepsilon_{n,\bk+\bh(\bk)}$ avoids all complex numbers $z$ in
the path of the analytic continuation we are interested in. As
mentioned before, when $z$ approaches the real axis, the integrand
concentrates on the Fermi surface. Fix $E \in \R$, and assume
that we are interested in the analytic continuation on a line
descending from the upper complex plane and passing through $E$.
Assume that $n_{0},\bk_{0}$ are such that
$\varepsilon_{n_{0}\bk_{0}} = E$. Then, near $\bk_{0}$, we have
\begin{align*}
  \varepsilon_{n_{0}\bk} \approx E + \nabla\varepsilon_{n_{0}\bk_{0}} \cdot (\bk-\bk_{0})
\end{align*}
and therefore, for small $\bh$ and $\bk$ close to $\bk_{0}$:
\begin{align*}
  {\rm Im} (\varepsilon_{n_{0}\bk+i\bh(\bk)}) \approx \nabla\varepsilon_{n_{0}\bk_{0}} \cdot \bh(\bk)
\end{align*}
By choosing $\bh(\bk)$ to be oriented in the direction of
$-\nabla \varepsilon_{n_{0}\bk_{0}}$ in a neighborhood of $\bk_{0}$,
we can ensure that $\varepsilon_{n_{0}\bk}$ has a negative imaginary
part whenever its real part becomes close to $E$. When
$\varepsilon_{n\bk}$ is far from $E$, we should set $\bh(\bk)$ to
zero. We use the general form
\begin{align}
  \label{eq:ki_k}
  \bh(\bk) = - \sum_{n \in 1}^{M} \alpha \nabla \varepsilon_{n\bk}\, \chi\left( \frac{\varepsilon_{n\bk} - E}{\Delta E } \right)
\end{align}
where $\alpha$ and $\Delta E $ are  constants with $\alpha>0$, and $\chi(E)$ is a
cutoff function, equal to $1$ at $0$ and zero for large values (in
practice, we use a Gaussian). 

To summarize, we employ as a definition of our BCD transformation the following function:
\begin{align*}
\bkappa(\bk; \alpha, E, \Delta E) \equiv \bk - i  \alpha\sum_{n=1}^{M} \nabla \varepsilon_{n\bk}\, \chi\left( \frac{\varepsilon_{n\bk} - E}{\Delta E } \right)
\end{align*}

The appropriate choice of the parameters is not trivial. The parameter
$\alpha$ scales the whole deformation. It
should ideally be chosen large enough for the deformation to be
effective, but choosing it too large invalidates the first-order
expansion above. The parameter $\Delta E $ must be chosen small enough 
to restrict the deformation to the surrounding of $\varepsilon_{n\bk}$, 
but choosing it too small results in rapid variations of h and requires 
a finer discretization of the Brillouin zone to integrate correctly.

For $\bh(\bk)$ to be a smooth function,
$\nabla\varepsilon_{n\bk}$ should be smooth whenever
$\chi(\tfrac{\varepsilon_{n\bk}-E}{\Delta E })$ is nonzero. When $\Delta E $
is small, this is possible under the first condition outlined in
Section~\ref{sec:green_function_periodic_hamiltonian}: near the Fermi
surface $S(E)$, bands should not cross, so that
$\nabla \varepsilon_{n\bk}$ is smooth. In order for the deformation to
produce its expected effect, $\nabla \varepsilon_{n\bk}$ should not be
small at the Fermi level, recovering the second condition in
Section~\ref{sec:green_function_periodic_hamiltonian}.

More quantitatively, we collect the conditions that $\alpha$, $\Delta E $ and
the discretization parameter $N$ should satisfy in order to ensure
a good approximation of the analytic continuation of
$R_{0}(\bR,\bR';z)$ near an energy $E$:
\begin{itemize}
    \item To ensure a smooth $\bh$,
    \begin{align*}
        \Delta E &\ll{\rm{dist}}(E, z) \quad \mbox{for all van Hove singularities $z$.}
    \end{align*}
    \item To ensure a valid first-order approximation,
    \begin{align*}
        \alpha | \nabla \varepsilon_{n\bk}| & \ll {\rm{diam}}(\mathcal{B}) \quad \forall \bk \in S_{n}(E)
    \end{align*}
    \item To ensure $z$ remains above the deformed spectrum at first
      order, when ${\rm Im}(z) < 0$,
    \begin{align*}
      |\rm{Im}(z)| & \ll \alpha  |\nabla \varepsilon_{n\bk}| ^2 \quad \forall \bk \in S_{n}(E)
    \end{align*}
    \item To ensure an accurate integration,
    \begin{align*}
        \frac{\rm{diam}(\mathcal{B})}{N} & \ll \min\left(\frac{\Delta E }{|\nabla \varepsilon_{n\bk}|}, \alpha  |\nabla \varepsilon_{n\bk} |\right)\quad \forall \bk \in S_{n}(E)\\
        N & \gg |\bR-\bR'|
    \end{align*}
  \end{itemize}
 
  Even though good results might be obtained even violating these
  conditions, they give useful rules of thumb to choose the parameters
  for a given system. Note that at van Hove singularities it is
  impossible to continue the resolvent. Therefore the procedure
  outlined above is only applicable outside of singularities, and we
  will seek resonances there.

\section{Applications}
\label{sec:apps}
We now apply the  method developed above to local perturbations of
tight-binding Hamiltonians.


\subsection{1D tight-binding for a diatomic chain}

\subsubsection{Perfect crystal}
Let us now consider a one dimensional chain containing two types of
atoms alternatively. The site energies of the atoms are $E_a$ and $E_b$, and the electron can hop from one site to another with a hopping constant 1:
\begin{align*}
    H_0&=\begin{pmatrix}
    \ddots &\ddots&\ddots\\
    & 1 &E_a & 1   \\
    &&1 & E_b & 1 \\
    &&&1&E_a & 1   \\
    &&&&1 & E_b & 1 \\
    &&&&&\ddots & \ddots  & \ddots
    \end{pmatrix}
  \end{align*}
The Bloch
transform of $H_{0}$ is
\begin{align*}
    H_{0,k}&=\begin{pmatrix} E_a & e^{-ik}+1 \\
    e^{ik}+1 &E_b 
    \end{pmatrix}
\end{align*}
with bands
\begin{align*}
     \varepsilon_{\pm,k}&
    =\frac{E_a+E_b}{2}\pm \sqrt{\frac{(E_a-E_b)^2}{4}+4\cos^2\left(\frac{k}{2}\right)}
\end{align*}
In numerical experiments, we choose $E_a=1$, $E_b=0$, so that the
continuous spectrum of $H_{0}$ is  approximately $[-1.56, 0] \cup [1, 2.56]$.



We can apply the method of the previous section to obtain the analytic
continuation of the $2\times 2$ matrix $R_{0}(\bR,\bR';z)$. We show in
Figure~\ref{fig:deformation_dia} the BCD we use for
the energy $E=2$.

\begin{figure}[h!]
\centering
    \begin{subfigure}[t]{0.45\textwidth}
        \centering
        \includegraphics[width=\textwidth]{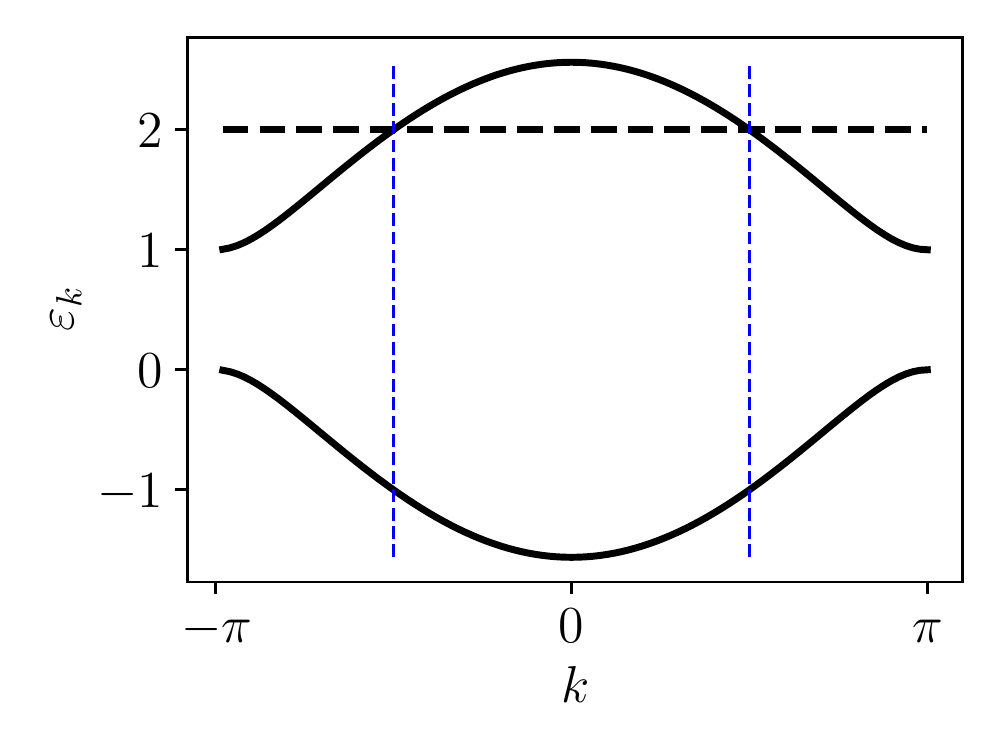}
         \caption{Band structure $\varepsilon_{\bk}$.}
     \end{subfigure}
    \quad
    \begin{subfigure}[t]{0.45\textwidth}
        \centering
         \includegraphics[width=\textwidth]{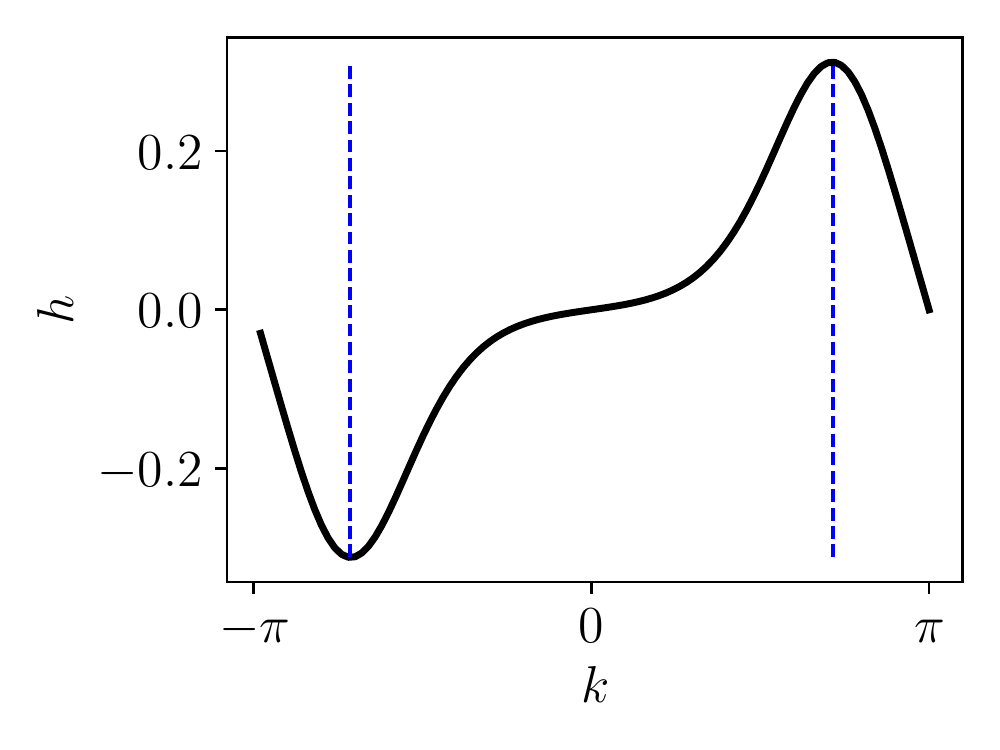}     
\caption{Profile of $\bh$ from \eqref{eq:ki_k}.}
     \end{subfigure}
     \caption{BCD at the energy of interest $E=2$ for the diatomic chain
       ($\alpha=0.4$, $\Delta E =0.7$, $\alpha=1$). The contour is deformed in
       the direction $-\nabla \varepsilon$ near the Fermi surface $S(E)$
       (dashed blue lines) at energy $E$ (dashed black line).}
    \label{fig:deformation_dia}
 \end{figure}

We apply it in Figure~\ref{fig:map_green_diatomic} to compute the continuation of the
trace per unit cell of the Green function
\begin{align*}
  {\rm Tr}(R_0(0,0;z)) = \frac 1 {|\BZ|}\int_{\BZ}{\rm Tr}\left( \frac 1 {z - H_{0,\bk}} \right) d\bk
\end{align*}
whose imaginary part for real $z$ is equal to $- \pi $ times the density of
states. The BCD moves the discretized continuum of poles further down
in the complex plane, allowing us to compute the continuation of
the resolvent across the spectrum of $\hat H_{0}$.

\begin{figure}[h!]
    \centering
    \includegraphics[width=0.5\textwidth]{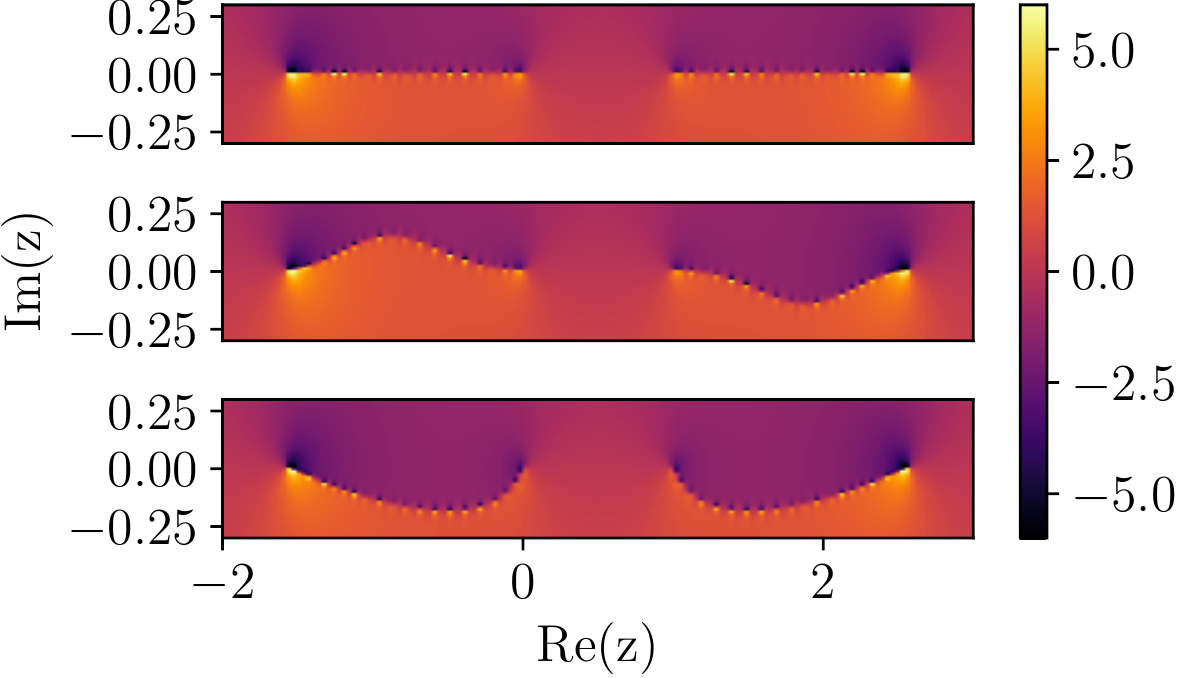}
    \caption{Imaginary part of the trace by cell of the Green function
      for the diatomic model: with no deformation (top), with a BCD with $E=2$
      (middle), and with  $E={\rm Re}(z)$ for every $z$ (bottom). Parameters are $N=50$, $\alpha=0.3$, $\Delta E =0.5$. }
    \label{fig:map_green_diatomic}
\end{figure}

\subsubsection{Defect}
To obtain resonances, we now perturb the system locally in the following manner:
\begin{align}
\label{eq:defect}
 	H=H_0+V&= \begin{pmatrix}
     &  \ddots \\
     \ddots & \ddots & 1 \\
     & 1 & E_a & \epsilon \\
     &    &\epsilon & E_b & 1 \\
     &    &         & 1& E_a &1 \\
     &    &    &     & 1& E_b &1 \\
     &    &    &   && 1& E_a &\epsilon \\
     &    &    &   & &  &\epsilon& E_b &1 \\
     &    &  &  &  & &  &1&\ddots &\ddots \\
     &  &  &  &  &  & &  &\ddots& &\\
    \end{pmatrix}
\end{align}
This defect is chosen so that, at $\epsilon=0$, there is a central
region unconnected to the rest of the chain, with four eigenvalues ($-1.19$, $-0.29$, $1.29$, $2.19$), embedded in the continuous spectrum
of $H_{0}$. At $\epsilon > 0$, the embedded eigenvalues turn into
resonances. We show this by plotting the smallest singular value of
$\hat 1
- \hat V \hat R_{0}(z)$ as a function of
$z$ in Figure~\ref{fig:poles_diatomic}, for $\epsilon=0.2$. We find
resonances close to the eigenvalues at $\varepsilon=0$, as expected.


\begin{figure}[h!]
        \centering
         \includegraphics[width=.7\textwidth]{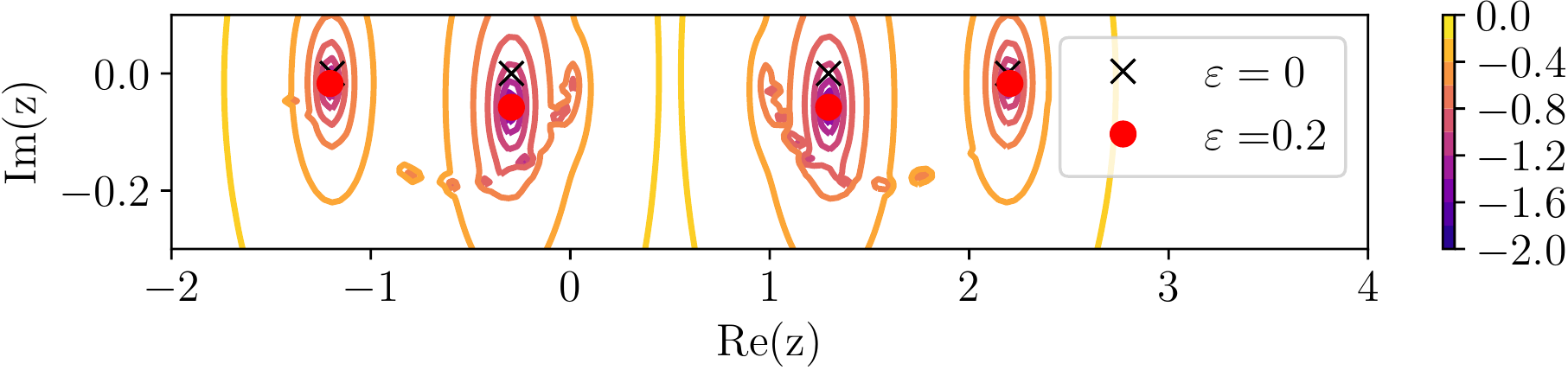} 
         \caption{Base 10 logarithm of the smallest singular value of
      $\hat 1-\hat V\hat R_0(z)$, with the same parameters as in Figure~\ref{fig:map_green_diatomic}. The black crosses are eigenvalues at $\varepsilon=0$,
      the red dots are resonances at $\varepsilon=0.2$
    }
    \label{fig:poles_diatomic}
 \end{figure}

For this simple monodimensional model, computing resonances can be
done using an explicit computation of the Green function (or of the
scattering matrix) through Schur complements. We verified that the
results given by our method on a simple case are in agreement with the
resonances found with scattering matrices \cite{hatano}.

  \subsection{Graphene}
\label{sec:graphene}

\subsubsection{Perfect crystal}
In this section, we study the standard nearest-neighbour model of
graphene with hopping parameter $t$, and lattice vectors
$\mathbf{a}_{1} = (\sqrt{3}/2, 1/2), \mathbf{a}_{2} = (\sqrt{3}/2, -1/2)$. The Bloch transform of the
Hamiltonian writes:
\begin{align*}
    H_{\bk}=\begin{pmatrix} 0 & -t(1+e^{i \bk \cdot \mathbf{a_1}}+ e^{i \bk\cdot  \mathbf{a_2}}) \\
    -t(1+e^{-i \bk \cdot \mathbf{a_1}}+ e^{-i \bk \cdot \mathbf{a_2}}) & 0
    \end{pmatrix}
\end{align*}
The Brillouin zone is a unit cell of the lattice spanned by the
reciprocal vectors
${\mathbf b}_{1} = 2\pi(1/\sqrt 3, 1), {\mathbf b}_{2} = 2\pi(1/\sqrt
3, -1)$. In the following, the plots in the Brillouin zone will be
given in the reduced coordinate system
$\bk = k_{1} {\mathbf b}_{1} + k_{2} {\mathbf b}_{2}$, in which we
take the Brillouin zone as $[-1/2, 1/2]^{2}$. We take $t=1$ throughout.

The dispersion relation is represented Figure
\ref{fig:surface_graphene}. 
\begin{figure}[h!]
    \centering
    \begin{subfigure}[t]{0.45\textwidth}
    \centering
    \includegraphics[width=\textwidth]{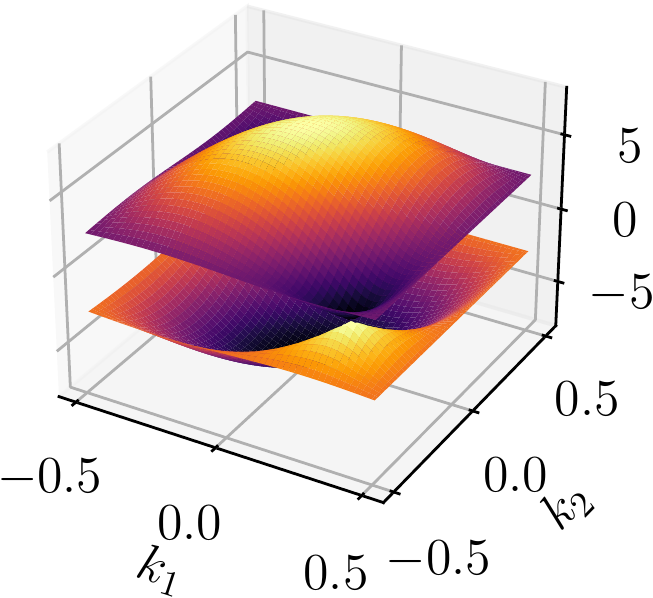}
    \caption{Dispersion relation $\varepsilon(\bk)$ of graphene.}
    \label{fig:surface_graphene}
    \end{subfigure}
    \hfill
    \begin{subfigure}[t]{0.45\textwidth}
    \centering
    \includegraphics[width=\textwidth]{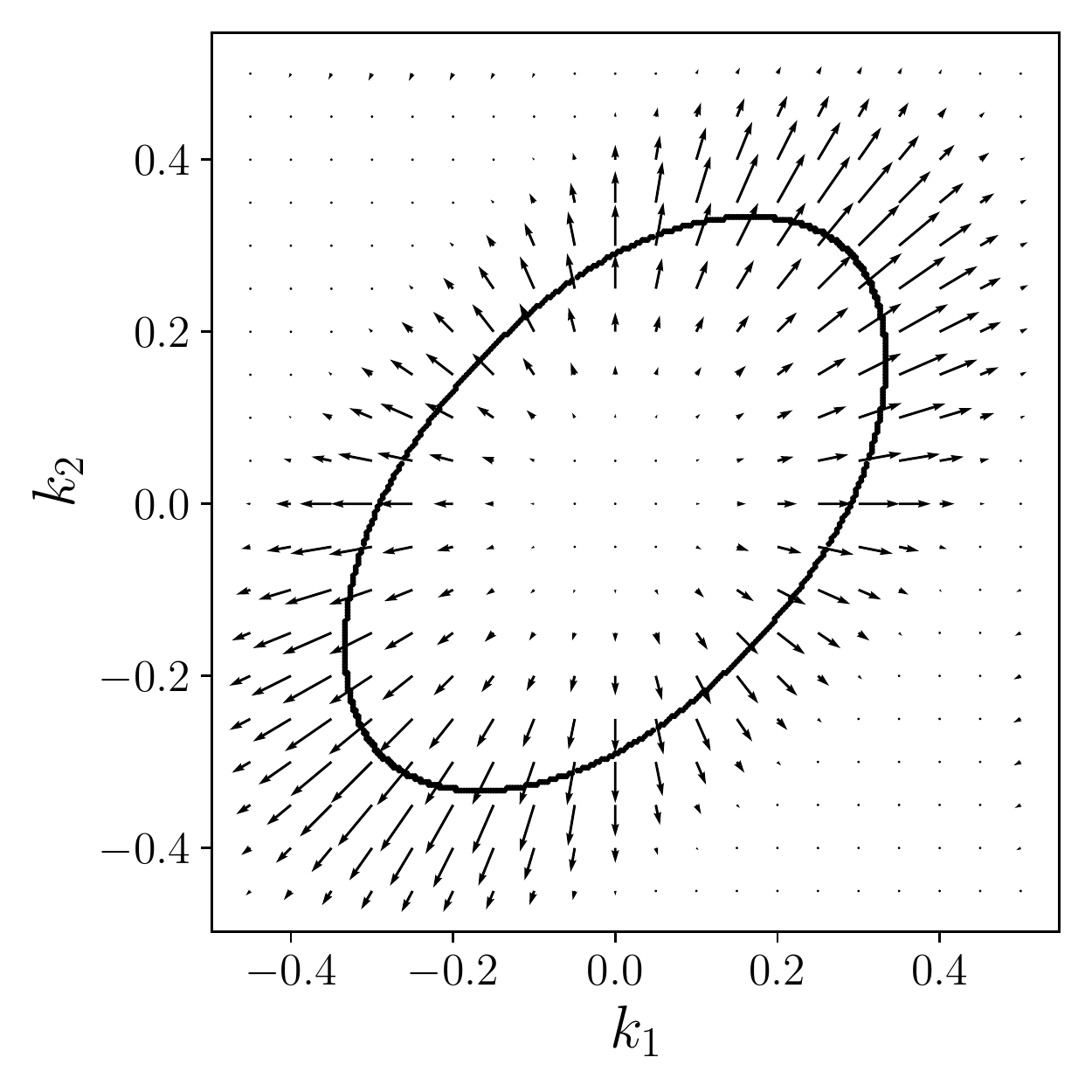}
    \caption{$\bh(\bk)$ at $E=2$. The black line is the Fermi surface.}
    \label{fig:ki_graphene}
    \end{subfigure}
    \caption{Dispersion relation and BCD for the graphene.}
\end{figure}

Our BCD, represented in Figure~\ref{fig:ki_graphene} for $E=2$, allows
us to compute the Green function (Figure~\ref{fig:map_green_graphene})
and therefore the density of states (Figure~\ref{fig:DOS_graphene}). 
The density of states (DOS) is defined as 
\begin{align*}
  D(z)&=-\frac{1}{\pi}\lim_{\eta \rightarrow 0^+}{\rm Im}\left(\underline{{\rm Tr}}\left((z+i\eta -H)^{-1}\right)\right)
\end{align*}
where $\underline{\rm{Tr}}$ is the trace by cell of the operator.

The dispersion relation is singular at the two Dirac points, at which
the multiplicity of the eigenvalues is double and where the behaviour
of the eigenvalues is linear at first order. The BCD
method is efficient for the $E$ which are not too close to these
singularities or to the other van Hove points, where the gradient
vanishes ($E=0, \pm t, \pm 3t$); the BCD is unable to move these
singularities away from the real axis.

We compare the DOS to both the exact analytic formula from \cite{2014}
and the standard approximation
\begin{align*}
  D_{N,\eta}(E) \approx \frac 1 {N^{2}} \sum_{\bk \in \BZ_{N}}\sum_{n=1}^{M} g\left( \frac{\varepsilon_{n\bk} - E}{\eta} \right)
\end{align*}
with $g$ a normalized Gaussian and $\BZ_{N}$ a Monkhorst-Pack grid with
$N^{2}$ points. This method converges to the DOS as $N \to \infty$
then $\eta \to 0^{+}$ (but not the reverse). Optimizing $\eta$ as a
function of $N$ gives a convergence of the DOS as $1/N^{2}$ outside of
the van Hove singularities
\cite{dupuy2021finite,cances:hal-01796582,colbrook2021computing}. By
contrast, the BCD method directly gives an exponential convergence as a
function of $N$. Even for small values of $N$, we observe that the BCD
gives more accurate (but less smooth) results.

\begin{figure}[h!]
    \centering
    \includegraphics[width=\textwidth]{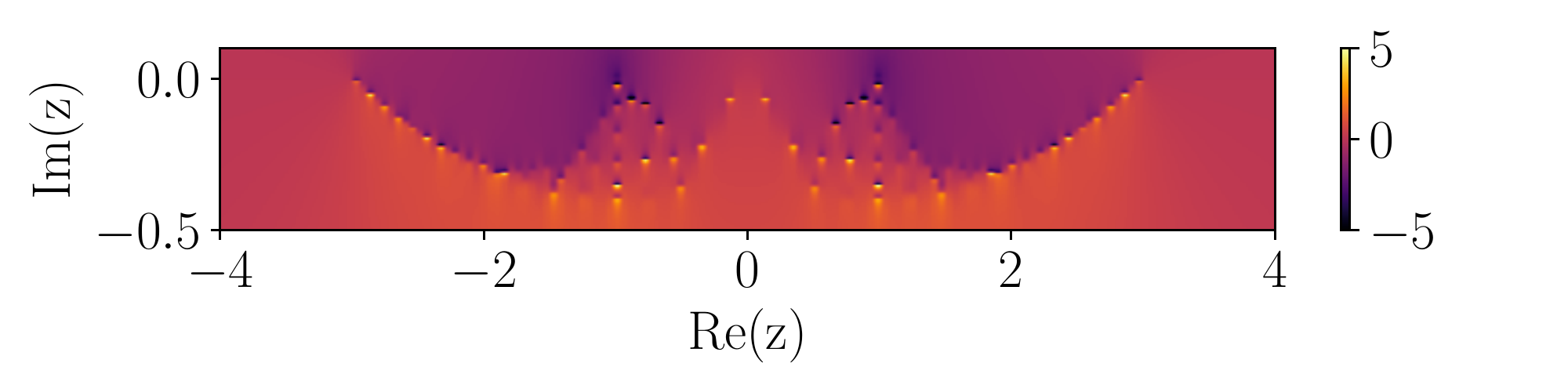}
    \caption{Imaginary part of the trace per unit cell of the Green
      function for the nearest-neighbor model of graphene. Parameters
      are $N=13$, $\alpha=0.4$, $\Delta E =0.5$.}
    \label{fig:map_green_graphene}
\end{figure}
\begin{figure}[h!]
    \centering
    \includegraphics[width=\textwidth]{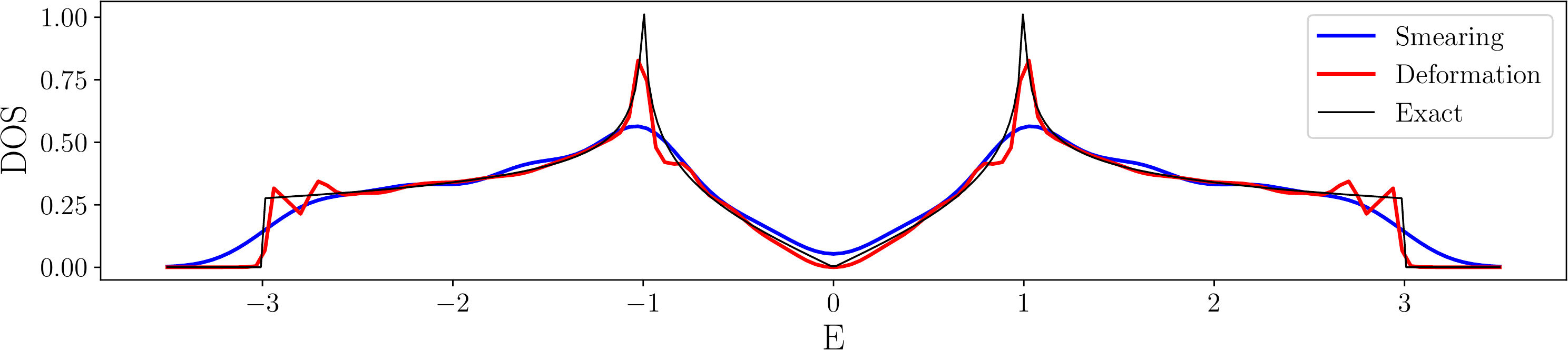}
    \caption{DOS of the nearest-neighbor model of graphene. The red
      plot is computed with our BCD method ($N=9$, $\alpha=0.3$,
      $\Delta E =0.4$), and the black plot is the reference DOS. The
      blue curve is the DOS obtained by the Gaussian
      smearing method for the same $N=9$, with an optimized value of
      $\eta=0.3$.}
    \label{fig:DOS_graphene}
\end{figure}

\subsubsection{Defect}
We now put an additional atom (adatom) on the surface of the
graphene as a perturbation. We implement this by a new site, linked to
a single site of the lattice with a hopping constant $\epsilon$, and a
site energy of $E_d$. This is similar to the "top" configuration of
\cite{PhysRevB.97.075417}. With a small hopping constant, we expect a
resonance, since the configuration $\epsilon=0$ gives a bound state at
energy $E_d$. We select the defect so that the resonance will be close
to $E=2$ (away from van Hove singularities). Taking $\epsilon=0.4$,
$E_{d}=2$, we expect a resonance with a real part close to $2$; the
Fermi Golden rule can be used to estimate the imaginary part to second
order in $\varepsilon$:
\begin{align*}
  {\rm Im}(z) \approx \varepsilon^{2} R_{0}(0,0; E_{d}+i0^{+})_{11}.
\end{align*}
We use the method BCD in the previous section (with the same
parameters $\alpha= 0.4, \Delta E = 0.5$) to compute $R_{0}$, and
obtain approximately ${\rm Im}(z) \approx -0.0854i$.

To compute the resonance exactly, we proceed using the same method as
before, solving the equation \eqref{eq:phi}. Since the defect $\hat V$
only links one site of the lattice to the adatom, we only need to
compute one coefficient of the Green function $R_0$. We look for zeros
of $\hat 1-\hat V \hat R_0$ in Figure~\ref{fig:svd_resonances_graphene}, and find a pole near
$2.062-0.0858i$, which is consistent with the Fermi Golden rule. The
convergence of the method (Figure~\ref{fig:convergence_graphene}) is
exponential with respect to $N$. We plot the resonance state in
Figure~\ref{fig:eigenstate_graphene}.

\begin{figure}[h!]
\centering
    \includegraphics[width=\textwidth]{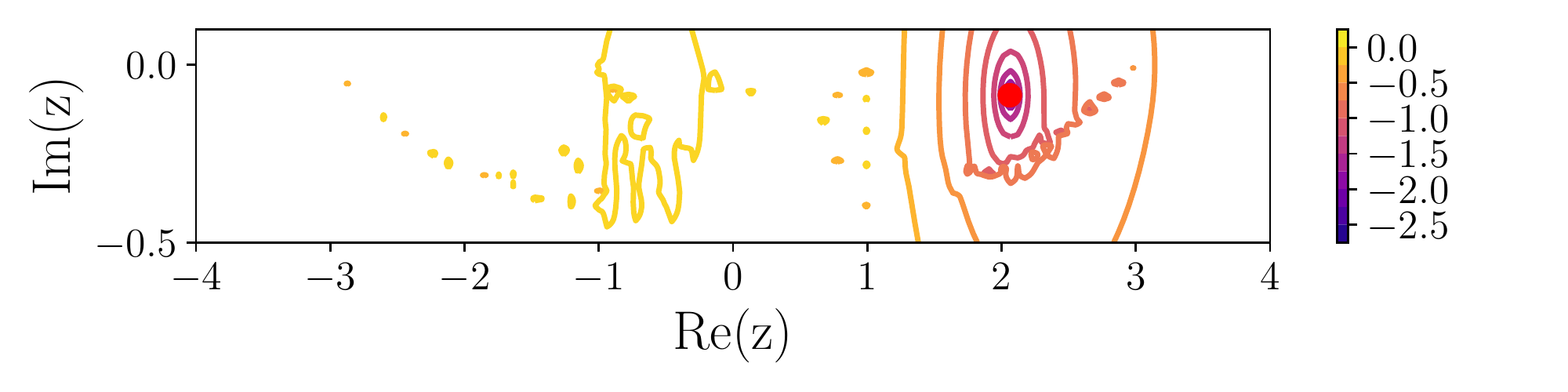}
    \caption{Base 10 logarithm of the smallest singular value of $\hat
      1-\hat V \hat R_0(z)$ for an adatom on graphene.}
    \label{fig:svd_resonances_graphene}
\end{figure}

\begin{figure}[h!]
    \centering
    \includegraphics[width=0.45\textwidth]{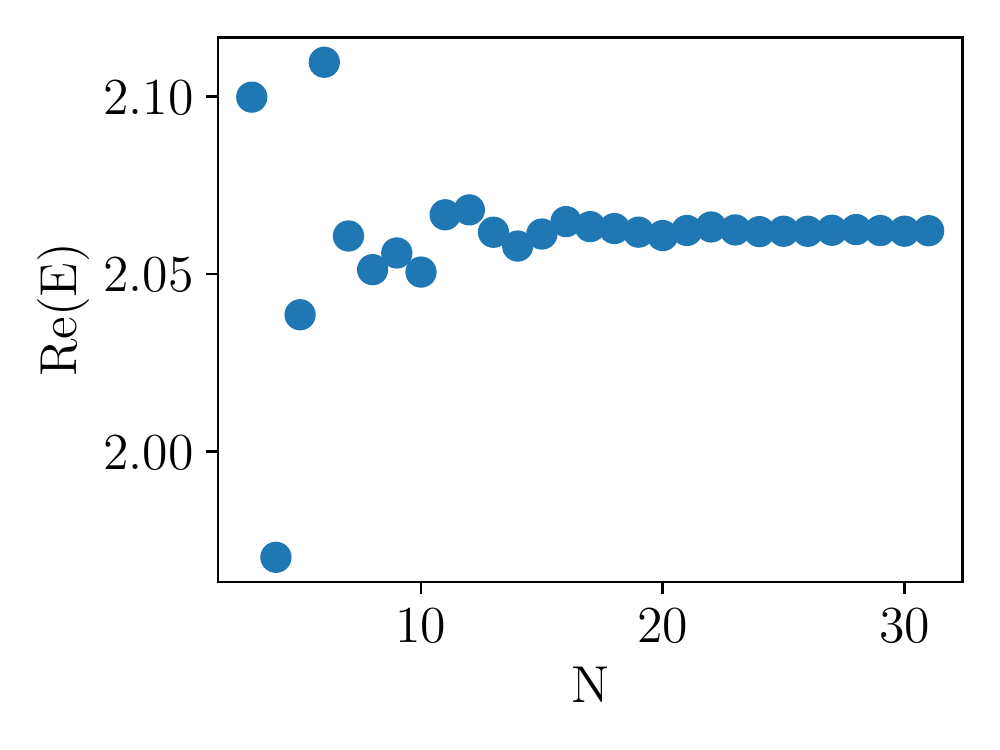} \quad
    \includegraphics[width=0.45\textwidth]{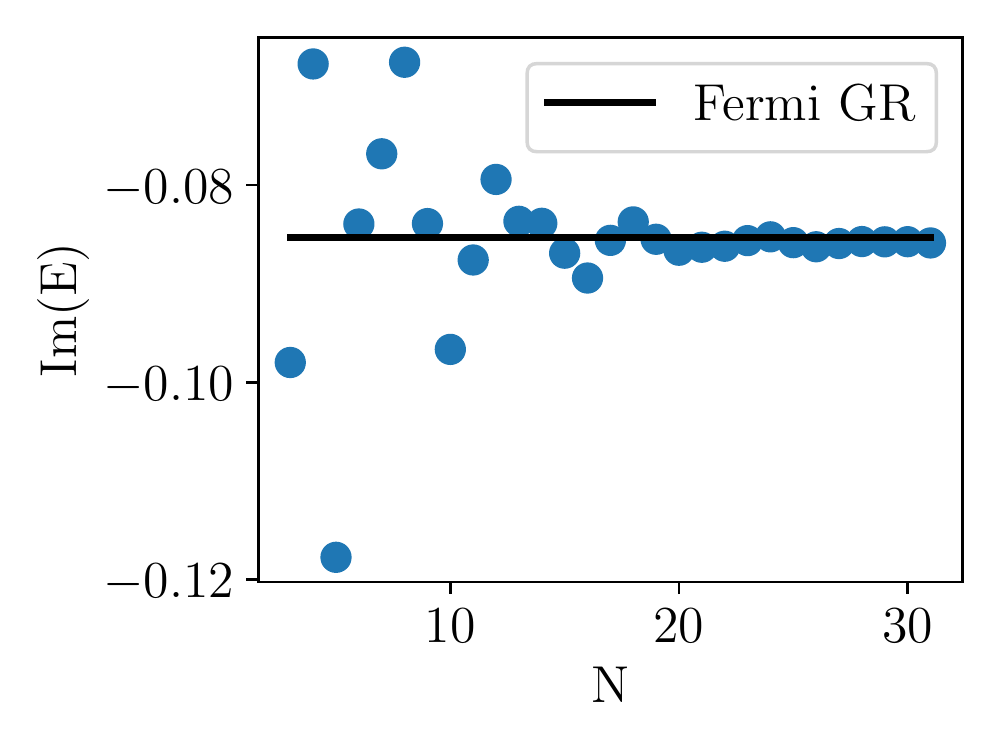}
    \caption{Convergence of the real and imaginary parts of the
      resonance energy with respect to the discretization parameter of
      the Brillouin zone $N$. The prediction of the Fermi Golden rule
      for the imaginary part of the resonance is given in black on the
      right panel.}
    \label{fig:convergence_graphene}
\end{figure}

\begin{figure}[h!]
    \centering
    \includegraphics[width=0.8\textwidth]{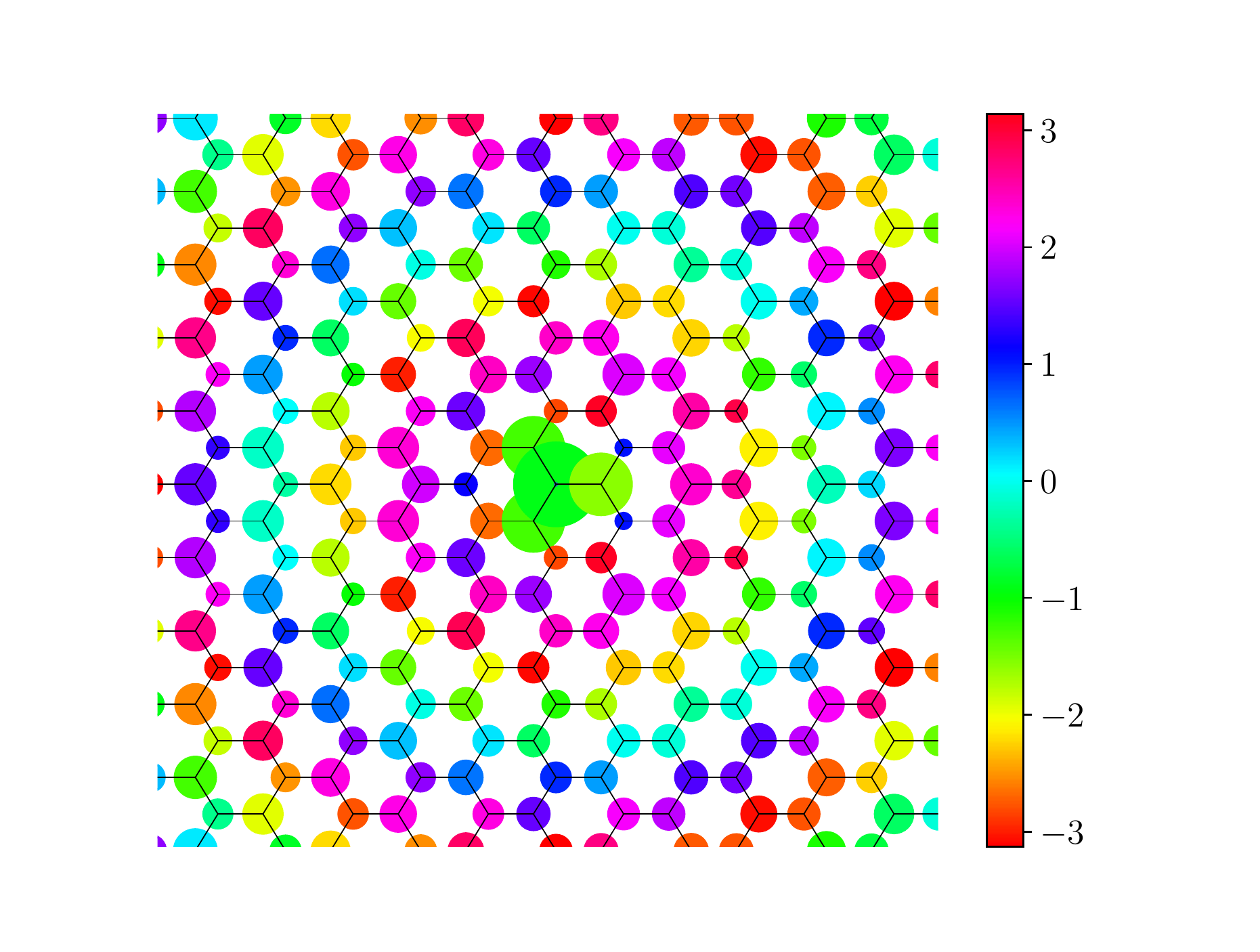}
    \caption{Resonant state. The size of the dots is proportional to the modulus of the state, the color represents the phase.}
    \label{fig:eigenstate_graphene}
\end{figure}
\section{Conclusion}
We have introduced a new method to find resonances in solids in the
framework of a localized, single defect, and demonstrated its
efficiency and generality. The method is based on a reformulation as
an integral equation, posed on the support of the defect. The kernel
of this integral equation is the Green function of the periodic
system, which we compute using a complex deformation of the Brillouin
zone. This algebraic (rather than geometric) splitting is very general
and applies to every locally perturbed periodic problem. Unlike
methods based on introducing artificial dissipation, the BCD-based
method is an exact reformulation of the problem, and only needs to be
converged with respect to the Brillouin zone discretization parameter
$N$ (and, in the case of defects with non-compact support, the domain
truncation parameter). We demonstrated in the Appendix A that this
method can be interpreted as a natural generalization of the
complex-scaling method to non-parabolic dispersion relations.

Although we demonstrated only one- and two-dimensional tight-binding
examples, the method fully applies to three-dimensional models
described by continuous Hamiltonians (for instance, those arising in
Density Functional Theory). Such an extension could be performed
directly using an iterative method to compute the nonlinear eigenvalue
problem \eqref{eq:res_expansion}, together with an iterative method
to express the action of the unit cell resolvent $1/(z-H_{0,\bk})$
appearing in the computation of $R_{0}$. This still requires a large
number of computations if a dense sampling of the Brillouin zone is
necessary. In this case, an alternative, more economical method would
be to use Wannier functions to reduce to a tight-binding model, either
exploiting periodicity \cite{marzari2012maximally}, or using linear
scaling approaches~\cite{BigDFT2020}. We hope to use these tools in
future work to compute resonances in response functions of realistic
solids (for instance, arising from time-dependent functional theory or
GW methods).

\vspace{6pt} 
\section*{Appendix A: connection with complex scaling}
We reformulate our method to make the connection to complex scaling
more explicit. Again for simplicity we consider a tight-binding model
on a lattice $\Gamma$, with state space $\ell^{2}(\Gamma, \C^{M})$, but
our discussion generalizes immediately to continuous models.
Each state $\psi = \{\psi_{n}\}_{n \in \Gamma}$ in this space can be
equivalently seen as a function of the pseudo-momentum $\bk \in \BZ$
through the Bloch (Fourier) transform, which we denote by
$\{\psi(\bk)\}_{\bk \in \BZ}$. We define a (non-unitary) transformation $U$ in (a subset
of) $\ell^{2}(\Gamma, \C^{M})$ by the formula
\begin{align*}
  ( U \psi)(\bk) = \psi(\bk+i \bh(\bk))
\end{align*}
for $\bk \in \BZ$. Letting ${\widetilde H_0} = U H_0 U^{-1}$, we note that
\begin{align*}
  (U H_0 U^{-1} \psi)(\bk) = (H_0 U^{-1} \psi)(\bk+i \bh(\bk)) = H_{0,\bk+i \bh(\bk)} (U^{-1} \psi)(\bk+i \bh(\bk)) = H_{0,\bk+i \bh(\bk)} \psi(\bk).
\end{align*}
This allows us to compute analytic continuations of the Green function
through the formula
\begin{align*}
  \langle  \psi_{1} | \frac 1 {z - \hat H_0} | \psi_{2} \rangle &= \langle  \hat U^{-*}\psi_{1}| \frac 1 {z- \hat U \hat H_0 \hat U^{-1}}  | \hat U \psi_{2} \rangle\\
  &= \frac 1 {|\BZ|} \int_{\BZ} \overline{\psi_{1}}(\bk+ i \bh(\bk)) \left(U^{-1} \frac 1 {z- U H_0 U^{-1}} U \psi_{2}\right)(\bk+i \bh(\bk)) \det(1+\bh'(\bk)) d\bk \\
  &= \frac 1 {|\BZ|} \int_{\BZ} \overline{\psi_{1}}(\bk+ i \bh(\bk)) \frac 1 {z- H_{0,\bk+i \bh(\bk)}}  \psi_{2}(\bk+i \bh(\bk)) \det(1+\bh'(\bk)) d\bk
\end{align*}
from where the results of Section~\ref{sec:periodic_integral_computation} follow.

We can now specialize this discussion to the case of the $d-$dimensional Hamiltonian
$H_{0} = -\tfrac 1 2 \Delta$, which can be seen as a periodic operator
with Brillouin zone $\BZ = \R^{d}$, and a single parabolic band
$\varepsilon_{\bk} = \tfrac 1 2 |\bk|^{2}$. In this case, the localization
function $\chi$ is unnecessary; omitting it from \eqref{eq:ki_k}
results in $\bh(\bk) = -\alpha \bk$, defining the transformation
$(U \psi)(\bk) = \psi(\bk- i \alpha \bk)$. It is instructive to compare this
to the classical complex scaling transformation, defined in real space
by $(U_{\theta} \psi)(e^{i \theta} x)$. Equivalently, this operator
can be defined in Fourier space by
$(U_{\theta} \psi)(\bk) = e^{-i\theta} \psi(e^{-i\theta} \bk)$. Omitting
the phase factor $e^{-i\theta}$ (which does not change
the results) and expanding to first order in $\theta$, this becomes
$(U_{\theta} \psi)(\bk) \approx \psi(\bk - i\theta \bk)$, which is identical
to the above with $\theta=\alpha$. The higher-order deviations are
inconsequential, and come
from the slightly different functional form of the complex
deformation: complex scaling originates from a group action, whereas
we do not find this requirement necessary in our deformation.

The above argument shows that the BCD method can be interpreted as
natural generalization of the complex-scaling method to non-parabolic
dispersion relations. Such a transformation in the reciprocal space is
reminiscent of the generalized complex transformation, also named
smooth exterior complex scaling, employed in molecular systems to
solve some of the problems arising in the uniform complex
scaling~\cite{genovese2015identification}. One important difference
between our implementation of the method and complex scaling, however,
is that we do not transform the whole operator $\hat H = \hat H_{0} + \hat V$, but
rather use the transformation only on $\hat H_{0}$ to compute its Green
function, and deduce the Green function of $\hat H$ by a Dyson equation.
This loses the ``global'' properties (we obtain resonances as a
nonlinear eigenvalue problem, rather than a linear one in the case of
complex scaling), but is much more flexible (since we do not need to
transform $\hat V$, a possibly complicated operation). In simple cases
however, it is possible to compute the transformation $\hat U$ and the
non-Hermitian operator $\hat{\widetilde H} = \hat U \hat H \hat U^{-1}$ explicitly.

\section*{Appendix B: normalization of the resolvent }
Consider an analytic family of matrices $A(z)$, and choose $z_{0}$ such
that $A(z_{0})$ has a simple eigenvalue $0$. There are $v, w$ such that
$A(z_{0}) v = 0, A(z_{0})^{*} w = 0$. We fix $v$ and $w$ to be normalized
to $1$. This means that $A(z_{0}) =
\begin{pmatrix}
  0&0\\0&X
\end{pmatrix}
$
in the representation where the input space is split between $v$ and
$v^{\perp}$, and the output between $w$ and $w^{\perp}$. $X$ is
invertible. Therefore,
\begin{align*}
  A(z) \approx
  \begin{pmatrix}
    A_{11}(z-z_{0}) & A_{12}(z-z_{0})\\
    A_{21}(z-z_{0})&X
  \end{pmatrix}
\end{align*}
with $A_{11} = \langle w, A'(z_{0}) v\rangle$.
The Schur complement is $A_{11}(z-z_{0}) +
O((z-z_{0})^{2})$, with inverse $\frac 1 {z-z_{0}} A_{11}^{-1} + O(1)$, and
it follows that
\begin{align*}
  A(z)^{-1} &\approx \frac 1 {z-z_{0}}
  \begin{pmatrix}
    A_{11}^{-1} & 0\\
    0&0
  \end{pmatrix} + O(1)
  \\
  &= \frac 1 {z-z_{0}}\frac 1 {\langle  w, A'(z_{0}) v \rangle} |v\rangle\langle  w|
\end{align*}

Applying to $A(z) = 1 - V R_{0}(z)$, we have $V R_{0}(z_{0})
\phi = \phi$ so we identify $v = \phi$. For $w$ we solve $V R_{0}(z_{0})^{*}
\widetilde \phi = \widetilde \phi$ and so we identify $w$ to
$R_{0}(z_{0})^{*} \widetilde \phi$, solution of $R_{0}(z_{0})^{*} V w = w$. The result is
\begin{align*}
  R(z) = R_{0}(z) A(z)^{-1} \approx \frac 1 {z-z_{0}} \frac 1 {- \langle  R_{0}(z_{0})^{*} \widetilde \phi, V R_{0}'(z_{0}) \phi \rangle} |R_{0}(z_{0}) \phi\rangle \langle R_{0}(z_{0})^{*} \widetilde \phi|
\end{align*}

With $\psi(z_0) =R_{0}(z_{0})\phi(z_0)=\overline{w}$, the residue of $(1 - V 
R_{0}(z))^{-1}$ at $z_{0}$ can readily be seen to be
\begin{align*}
  (1 - V R_{0}(z))^{-1} \approx \frac1 {-\langle \overline{\psi(z_0)}|  V  R_{0}'(z_{0})| \varphi(z_0) \rangle} |\varphi(z_0)\rangle\langle \overline{\psi(z_0)}|\frac 1 {z-z_{0}}
\end{align*}
which gives the following condition to have \eqref{eq:res_expansion}:
\begin{align}
  \langle \overline{\psi(z_0)}| \hat V \hat R_{0}'(z_{0})| \varphi(z_0) \rangle = -1
\end{align}

Note that $\phi \neq \widetilde \phi$. If $H$ is real, then $R^{*}(z) =
\overline R(z)$ and therefore $\widetilde \phi = \overline{\phi}$ and
the residual is of the form $|\psi\rangle\langle \overline{\psi}|$.

\section*{Acknowledgments}
This project has received funding from the European Research Council
(ERC) under the European Union's Horizon 2020 research and innovation
programme (grant agreement No 810367). Stimulating conversations with
Mi-Song Dupuy and Sonia Fliss are gratefully acknowledged.

\bibliography{bibliographie}

\begin{thebibliography}{10}

\bibitem{dyatlov2019mathematical}
Semyon Dyatlov and Maciej Zworski.
\newblock {\em Mathematical theory of scattering resonances}, volume 200.
\newblock American Mathematical Soc., 2019.

\bibitem{Gamow}
G.~Gamow.
\newblock Zur quantentheorie des atomkernes.
\newblock {\em Zeitschrift f\"ur Physik A Hadrons and Nuclei}, 51:204--212,
  1928.
\newblock 10.1007/BF01343196.

\bibitem{PhysRevC.47.768}
Tore Berggren and Patric Lind.
\newblock {Resonant state expansion of the resolvent}.
\newblock {\em {Phys. Rev. C}}, {47}({2}):768--778, {Feb} {1993}.

\bibitem{Hatano2008}
Keita Sasada, Naomichi Hatano, and Gonzalo Ordonez.
\newblock {Resonant Spectrum Analysis of the Conductance of Open Quantum System
  and Three Types of Fano Parameter}.
\newblock {\em J.Phys.Soc.Jap.}, 80:104707, 2011.

\bibitem{Myo01051998}
Takayuki Myo, Akira Ohnishi, and Kiyoshi Kato.
\newblock {Resonance and Continuum Components of the Strength Function}.
\newblock {\em {Progress of Theoretical Physics}}, {99}({5}):801--817, {1998}.

\bibitem{Lind1994}
P.~Lind, R.~J. Liotta, E.~Maglione, and T.~Vertse.
\newblock {Resonant state expansions of the continuum}.
\newblock {\em {Zeitschrift f\"ur Physik A Hadrons and Nuclei}},
  {347}({4}):231--236, {1994}.

\bibitem{PhysRevLett.79.2026}
Oleg~I. Tolstikhin, Valentin~N. Ostrovsky, and Hiroki Nakamura.
\newblock {Siegert Pseudo-States as a Universal Tool: Resonances, $\mathit S$
  Matrix, Green Function}.
\newblock {\em {Phys. Rev. Lett.}}, {79}({11}):2026--2029, {Sep} {1997}.

\bibitem{MUGA2004357}
J.G. Muga, J.P. Palao, B.~Navarro, and I.L. Egusquiza.
\newblock Complex absorbing potentials.
\newblock {\em Physics Reports}, 395(6):357--426, 2004.

\bibitem{complex_scaling_3D}
Alessandro Cerioni, Luigi Genovese, Ivan Duchemin, and Thierry Deutsch.
\newblock Accurate complex scaling of three dimensional numerical potentials.
\newblock {\em The Journal of Chemical Physics}, 138(20):204111, 2013.

\bibitem{givoli2013numerical}
Dan Givoli.
\newblock {\em Numerical methods for problems in infinite domains}.
\newblock Elsevier, 2013.

\bibitem{bonnetbendhia:hal-01793511}
Anne-Sophie Bonnet-Ben~Dhia, Sonia Fliss, and Yohanes Tjandrawidjaja.
\newblock {Numerical analysis of the Half-Space Matching method with Robin
  traces on a convex polygonal scatterer}.
\newblock In {\em {Maxwell's equations}}. {De Gruyter}, 2018.

\bibitem{gerard}
Christian G\'erard.
\newblock Resonance theory for periodic schr\"odinger operators.
\newblock {\em Bulletin de la Soci\'et\'e Math\'ematique de France},
  118(1):27--54, 1990.

\bibitem{Hoang2011TheLA}
Vu~Hoang.
\newblock The limiting absorption principle for a periodic semi-infinite
  waveguide.
\newblock {\em SIAM J. Appl. Math.}, 71:791--810, 2011.

\bibitem{joly:hal-00977852}
Patrick Joly, Jing-Rebecca Li, and Sonia Fliss.
\newblock {Exact boundary conditions for periodic waveguides containing a local
  perturbation}.
\newblock {\em {Communications in Computational Physics}}, 1(6):945--973, 2006.

\bibitem{zhang2021numerical}
Ruming Zhang.
\newblock Numerical methods for scattering problems in periodic waveguides.
\newblock {\em Numerische Mathematik}, 148(4):959--996, 2021.

\bibitem{combes_asymptotic_1973}
J.~M. Combes and L.~Thomas.
\newblock Asymptotic behaviour of eigenfunctions for multiparticle
  {Schrödinger} operators.
\newblock {\em Communications in Mathematical Physics}, 34(4):251--270,
  December 1973.

\bibitem{reed1979i}
Michael Reed and Barry Simon.
\newblock {\em Methods of modern mathematical physics. {I}: Functional
  Analysis}, volume~3.
\newblock Elsevier, 1979.

\bibitem{guttel_tisseur_2017}
Stefan Güttel and Françoise Tisseur.
\newblock The nonlinear eigenvalue problem.
\newblock {\em Acta Numerica}, 26:1–94, 2017.

\bibitem{alloul}
Henri Alloul.
\newblock {\em Introduction to the Physics of Electrons in Solids}.
\newblock 01 2011.

\bibitem{monkhorst}
Hendrik~J. Monkhorst and James~D. Pack.
\newblock Special points for brillouin-zone integrations.
\newblock {\em Phys. Rev. B}, 13:5188--5192, Jun 1976.

\bibitem{trefethen}
Mohsin Javed and Lloyd~N. Trefethen.
\newblock A trapezoidal rule error bound unifying the {Euler}–{Maclaurin}
  formula and geometric convergence for periodic functions.
\newblock {\em Proceedings of the Royal Society A: Mathematical, Physical and
  Engineering Sciences}, 470(2161):20130571, January 2014.

\bibitem{cances:hal-01796582}
Eric Canc{\`e}s, Virginie Ehrlacher, David Gontier, Antoine Levitt, and Damiano
  Lombardi.
\newblock {Numerical quadrature in the Brillouin zone for periodic
  Schr{\"o}dinger operators}.
\newblock {\em {Numerische Mathematik}}, 144:479--526, January 2020.

\bibitem{hatano}
Naomichi Hatano, Keita Sasada, Hiroaki Nakamura, and Tomio Petrosky.
\newblock {Some Properties of the Resonant State in Quantum Mechanics and Its
  Computation}.
\newblock {\em Progress of Theoretical Physics}, 119(2):187--222, 02 2008.

\bibitem{2014}
N.~A. Pike and D.~Stroud.
\newblock Tight-binding model for adatoms on graphene: Analytical density of
  states, spectral function, and induced magnetic moment.
\newblock {\em Physical Review B}, 89(11), Mar 2014.

\bibitem{dupuy2021finite}
Mi-Song Dupuy and Antoine Levitt.
\newblock Finite-size effects in response functions of molecular systems.
\newblock {\em arXiv preprint arXiv:2102.09841}, 2021.

\bibitem{colbrook2021computing}
Matthew Colbrook, Andrew Horning, and Alex Townsend.
\newblock Computing spectral measures of self-adjoint operators.
\newblock {\em SIAM Review}, 63(3):489--524, 2021.

\bibitem{PhysRevB.97.075417}
Susanne Irmer, Denis Kochan, Jeongsu Lee, and Jaroslav Fabian.
\newblock Resonant scattering due to adatoms in graphene: Top, bridge, and
  hollow positions.
\newblock {\em Phys. Rev. B}, 97:075417, Feb 2018.

\bibitem{marzari2012maximally}
Nicola Marzari, Arash~A Mostofi, Jonathan~R Yates, Ivo Souza, and David
  Vanderbilt.
\newblock Maximally localized wannier functions: Theory and applications.
\newblock {\em Reviews of Modern Physics}, 84(4):1419, 2012.

\bibitem{BigDFT2020}
Laura~E. Ratcliff, William Dawson, Giuseppe Fisicaro, Damien Caliste, Stephan
  Mohr, Augustin Degomme, Brice Videau, Viviana Cristiglio, Martina Stella,
  Marco D’Alessandro, Stefan Goedecker, Takahito Nakajima, Thierry Deutsch,
  and Luigi Genovese.
\newblock Flexibilities of wavelets as a computational basis set for
  large-scale electronic structure calculations.
\newblock {\em The Journal of Chemical Physics}, 152(19):194110, 2020.

\bibitem{genovese2015identification}
Luigi Genovese, Alessandro Cerioni, Maxime Morinière, and Thierry Deutsch.
\newblock Identification of resonant states via the generalized virial theorem,
  2015.

\end{thebibliography}
\bibliographystyle{unsrt}
\end{document}